\documentclass[11pt]{amsart}
\pdfoutput=1

\usepackage{amssymb}
\usepackage{mathtools}
\usepackage{microtype}
\usepackage[vmargin=1in,hmargin=1.5in]{geometry}
\usepackage{enumerate}
\usepackage[numbers]{natbib}

\usepackage{aliascnt}
\usepackage{hyperref}
\hypersetup{
  pdftitle={The abstract Hodge--Dirac operator and its stable
    discretization},
  pdfauthor={Paul Leopardi and Ari Stern},
  pdfsubject={MSC 2010: 65N30 (58A14)},
  bookmarksopen=false
}

\newcommand{\infsup}[2]{\operatorname{inf\vphantom{sup}}\displaylimits_{#1}
  \operatorname{sup\vphantom{inf}}\displaylimits_{#2}}

\theoremstyle{plain} %--default
\newtheorem{theorem}{Theorem}[section]

\newaliascnt{lemma}{theorem}
\newtheorem{lemma}[lemma]{Lemma}
\aliascntresetthe{lemma}

\newaliascnt{corollary}{theorem}
\newtheorem{corollary}[corollary]{Corollary}
\aliascntresetthe{corollary}

\theoremstyle{definition}
\newaliascnt{definition}{theorem}
\newtheorem{definition}[definition]{Definition}
\aliascntresetthe{definition}
\newtheorem{example}[theorem]{Example}

\theoremstyle{remark}
\newtheorem*{notation}{Notation}
\newtheorem*{acknowledgments}{Acknowledgments}

\begin{document}

\title[The abstract Hodge--Dirac operator]{The abstract Hodge--Dirac
  operator\\and its stable discretization}

\author{Paul Leopardi}
\address{University of Newcastle\\
  Callaghan, NSW, Australia}
\email{paul.leopardi@newcastle.edu.au}

\author{Ari Stern}
\address{Washington University\\
  St.~Louis, Missouri, USA}
\email{astern@math.wustl.edu}

\begin{abstract}
  This paper adapts the techniques of finite element exterior calculus
  to study and discretize the abstract Hodge--Dirac operator, which is
  a square root of the abstract Hodge--Laplace operator considered by
  Arnold, Falk, and Winther [Bull. Amer. Math. Soc. 47 (2010),
  281--354].  Dirac-type operators are central to the field of
  Clifford analysis, where recently there has been considerable
  interest in their discretization.  We prove \emph{a priori}
  stability and convergence estimates, and show that several of the
  results in finite element exterior calculus can be recovered as
  corollaries of these new estimates.
\end{abstract}


%\date{\today}

\maketitle

\section{Introduction}

\subsection{Overview}
In the numerical analysis of elliptic PDEs, much attention has been
given (quite rightly) to the discretization of the second-order
Laplace operator.  The development of mixed finite elements (e.g.,
edge elements) paved the way for the discretization of other
Laplace-type second-order differential operators, such as the vector
Laplacian, with important numerical applications in computational
electromagnetics and elasticity.  The recent development of
\emph{finite element exterior calculus}
(\citet*{ArFaWi2006,ArFaWi2010}, extending earlier seminal work by
\citet*{Bossavit1988,Bossavit1998,Hiptmair1999,Hiptmair2002,Kotiuga1984,Nedelec1980,Nedelec1986},
and many others) has shown that these operators are special cases of
the \emph{Hodge--Laplace operator} on differential $k$-forms, which
can be stably discretized by certain families of finite element
differential forms. An even more general operator, called the
\emph{abstract Hodge--Laplace operator}, includes both the
aforementioned Hodge--Laplace operator on $k$-forms, as well as other
operators that arise, for example, in elasticity.

By comparison, \emph{Dirac-type operators} have received little
attention from the perspective of numerical PDEs---despite being, in
many ways, just as fundamental as the widely-studied Laplace operators
discussed above.  Informally, a Dirac operator is a square root of
some Laplace operator, and is therefore a \emph{first-order} (rather
than second-order) differential operator.  Dirac-type operators arise
both in analysis (\citet{EaRy2007}) and in differential geometry
(\citet{Friedrich2000}), in addition to their well-known, eponymous
origins in quantum mechanics (\citet{Dirac1928}).  The study of these
first-order operators is also associated with a number of celebrated
theorems, including the Atiyah--Singer index theorem \citep{AtSi1963},
Witten's proof of the positive energy theorem \citep{Witten1981}, and
the solution of the Kato square root problem (\citet{AxKeMc2006}).
\emph{Clifford analysis} is the study of Dirac operators in various
settings, including on smooth manifolds
(\citet{Delanghe2001,Cnops2002}).

Recently, there has been growing interest in developing a theory of
\emph{discrete Clifford analysis}, based on lattice discretizations of
Dirac operators (\citet{FaKaSo2007,Faustino2009,BrDeSoVa2009}).  In
many respects, this work resembles the various lattice approaches to
discretizing exterior calculus
(\citet{DeHiLeMa2005,Harrison2005,BoHy2006}), particularly in the use
of primal-dual mesh pairs.  These approaches are closer in spirit to
finite difference methods than to finite element methods, in that they
focus more on the degrees of freedom themselves than on basis
functions and interpolants.  Consequently, these methods tend to be
less amenable to stability and convergence analysis, or to
higher-order discretizations, compared with the finite element
exterior calculus of \citet{ArFaWi2006,ArFaWi2010}.

However, scant attention has been given to the possibility of using a
mixed finite element approach to discretize Dirac-type operators and
their associated first-order PDEs. The present paper aims to fill this gap.

The paper is organized as follows:
\begin{enumerate}
\item In the remainder of this section, we briefly provide some
  background on the relationship between exterior calculus and
  Clifford analysis, which we illustrate with examples on
  $\mathbb{R}^2$ and $\mathbb{R}^3$.

\item In \autoref{sec:nilpotent}, we develop an abstract version of
  Hodge theory and the Hodge--Dirac operator, and prove well-posedness
  for the associated variational problem. Following
  \citet{AxMc2004,AxKeMc2006}, this approach replaces the exterior
  derivative by an arbitrary nilpotent operator on a Hilbert
  space. Despite the simplicity of this setup---even simpler than the
  Hilbert complexes considered by \citet{BrLe1992} and
  \citet{ArFaWi2010}---it retains all of the salient features needed
  for the analysis, including abstract versions of the Hodge
  decomposition and Poincar\'e inequality.

\item In \autoref{sec:discrete}, we consider the discretization of the
  variational problem for the abstract Hodge--Dirac operator, proving
  stability and \textit{a priori} error estimates. Moreover, given
  some additional assumptions on the nilpotent operator, we develop
  improved estimates based on a duality/$ L ^2
  $-lifting/Aubin--Nitsche-type argument.

\item In \autoref{sec:laplace}, we relate the abstract Hodge--Dirac
  operator to its square, the abstract Hodge--Laplace operator, and
  show that some of the key results of finite element exterior
  calculus (cf.~\citet{ArFaWi2010}) are recovered as corollaries of
  the estimates in \autoref{sec:nilpotent} and \autoref{sec:discrete}.

\item In \autoref{sec:complex}, we briefly discuss the correspondence
  between the nilpotent operator formalism, considered here, and the
  Hilbert complex formalism used by \citet{ArFaWi2010}. In particular,
  this correspondence implies that the elements used in finite element
  exterior calculus, e.g., the $ \mathcal{P} _r $ and $ \mathcal{P} _r
  ^- $ families of piecewise-polynomial differential forms on
  simplicial meshes and the $ \mathcal{S} _r $ family on cubical
  meshes (cf.~\citet{ArAw2013}), also yield a stable discretization of
  the Hodge--Dirac problem.

\item Finally, in \autoref{sec:numerical}, we provide a numerical
  application for the discretization of the Hodge--Dirac operator:
  computing a vector field with prescribed divergence and curl.

\end{enumerate}

\subsection{Motivating examples: exterior calculus and Clifford
  analysis}

In this section, we illustrate the relationship between the Euclidean
Hodge--Dirac operator for differential forms and certain Dirac
operators for Clifford algebra-valued functions.  In particular, we
focus on the important examples of $ \mathbb{R}^2 $ and $ \mathbb{R}^3
$, where the Clifford algebra (or its even subalgebra) can be
identified with the complex numbers $ \mathbb{C} $ or the quaternions
$ \mathbb{H} $.

\begin{example}[$\mathbb{R}^2$  and the Cauchy--Riemann equations]
  Let $ \Omega ( \mathbb{R}^2 ) = \bigoplus _{ k = 0 } ^2 \Omega ^k (
  \mathbb{R}^2 ) $ denote the graded vector space of smooth
  differential forms on $\mathbb{R}^2$, where $ \Omega ^k (
  \mathbb{R}^2 ) $ denotes the subspace of $k$-forms for $ k = 0, 1, 2
  $. A general element $ u \in \Omega ( \mathbb{R}^2 ) $ has the form
  \begin{equation*}
    u = u _0 + u _1 \,\mathrm{d}x _1 + u _2 \,\mathrm{d}x _2 + u _{ 1
      2 } \,\mathrm{d}x _1 \wedge \mathrm{d} x _2 .
  \end{equation*} 
  Letting $ \mathrm{d} $ be the ($+1$-graded) exterior derivative on $
  \Omega ( \mathbb{R}^2 ) $ and $ \mathrm{d} ^\ast $ be its
  ($-1$-graded) $ L ^2 $-adjoint, we have
  \begin{align*}
    \mathrm{d} u &= \partial _1 u _0 \,\mathrm{d}x _1 + \partial _2 u
    _0 \,\mathrm{d}x _2 + ( \partial _1 u _2 - \partial _2 u _1 )
    \,\mathrm{d}x _1 \wedge \mathrm{d} x _2,\\
    \mathrm{d} ^\ast u &= - ( \partial _1 u _1 + \partial _2 u _2 )
    + \partial _2 u _{ 1 2 } \,\mathrm{d}x _1 - \partial _1 u _{ 1 2 }
    \,\mathrm{d}x _2 .
  \end{align*}
  The \emph{Hodge--Dirac operator} (or Hodge--de~Rham operator) on $
  \Omega ( \mathbb{R}^2 ) $ is defined by $ \mathrm{D} = \mathrm{d} +
  \mathrm{d} ^\ast $, so that
  \begin{multline*}
    \mathrm{D} u = - ( \partial _1 u _1 + \partial _2 u _2 ) +
    ( \partial _1 u _0 + \partial _2 u _{ 1 2 } ) \,\mathrm{d}x _1 +
    ( \partial _2 u _0 - \partial _1 u _{ 1 2 } ) \,\mathrm{d}x _2 \\
    + ( \partial _1 u _2 - \partial _2 u _1 ) \,\mathrm{d}x _1 \wedge
    \mathrm{d} x _2 .
  \end{multline*}
  Note that $ \mathrm{D} $ is not a graded operator with respect to
  the integer grading of $ \Omega ( \mathbb{R}^2 ) $, since $
  \mathrm{d} $ is $+1$-graded and $ \mathrm{d} ^\ast $ is
  $-1$-graded. However, if we introduce the $ \mathbb{Z} _2 $ grading
  $ \Omega ( \mathbb{R}^2 ) = \Omega ^+ (\mathbb{R}^2) \oplus \Omega
  ^- (\mathbb{R}^2) $, where
  \begin{alignat*}{2}
    \Omega ^+ ( \mathbb{R}^2 ) &= \bigoplus _{ k \text{ even} } \Omega
    ^k (\mathbb{R}^2) &&= \Omega ^0 (
    \mathbb{R}^2  ) \oplus \Omega ^2 ( \mathbb{R}^2  ) ,\\
    \Omega ^- ( \mathbb{R}^2 ) &= \bigoplus _{ k \text{ odd} } \Omega
    ^k ( \mathbb{R}^2 ) &&= \Omega ^1 (\mathbb{R}^2 ),
  \end{alignat*}  
  then $ \mathrm{D} $ is an odd-graded operator, mapping $ \Omega ^+
  (\mathbb{R}^2 ) $ to $ \Omega ^- (\mathbb{R}^2 ) $ and vice versa.
  Restricting the Hodge--Dirac operator to the even-degree part, whose
  elements are $ u ^+ = u _0 + u _{ 1 2 } \,\mathrm{d}x _1 \wedge
  \mathrm{d} x _2 \in \Omega ^+ ( \mathbb{R}^2 ) $, we obtain
  \begin{equation*}
    \mathrm{D} u ^+ = ( \partial _1 u _0 + \partial _2 u _{ 1 2 } )
    \,\mathrm{d}x _1 + ( \partial _2 u _0 - \partial _1 u _{ 1 2 } )
    \,\mathrm{d}x _2 .
  \end{equation*} 
  Taking $ i = - \mathrm{d}x _1 \wedge \mathrm{d} x _2 $, we can thus
  identify $ u ^+ $ with the complex-valued function $ f = u _0 - u _{
    1 2 } i $, and it follows that $ \mathrm{D} u ^+ = 0 $ if and only
  if
  \begin{equation*}
    \partial _1 u _0 + \partial _2 u _{ 1 2 } = 0, \qquad \partial _2
    u _0 - \partial _1 u _{ 1 2 } = 0,
  \end{equation*} 
  i.e., $f$ is a solution to the Cauchy--Riemann system.

  These observations have direct analogs in the language of Clifford
  analysis. The Clifford algebra $ \mathrm{Cliff} _{ 0, 2 }
  (\mathbb{R} ) $ on $\mathbb{R}^2$ has elements of the form
  \begin{equation*}
    a = a _0 + a _1 e _1 + a _2 e _2 + a _{ 1 2 } e _1 e _2 , 
  \end{equation*} 
  subject to the algebraic relations $ e _1 ^2 = e _2 ^2 = - 1 $ and $
  e _1 e _2 = - e _2 e _1 $. This is isomorphic to the algebra of
  quaternions $ \mathbb{H} $, while the even subalgebra $
  \mathrm{Cliff} _{ 0, 2 } ^+ (\mathbb{R} ) $, containing only
  even-degree terms $ a^+ = a _0 + a _{ 1 2 } e _1 e _2 $, is
  isomorphic to $\mathbb{C}$ by taking $ i = - e _1 e _2 $. Hence, $
  \Omega ( \mathbb{R}^2 ) $ can be identified with $ \mathrm{Cliff} _{
    0, 2 } (\mathbb{R} ) $-valued (i.e., quaternion-valued) functions
  on $\mathbb{R}^2$, by mapping $ u \mapsto u _0 + u _1 e _1 + u _2 e
  _2 + u _{ 1 2 } e _1 e _2 $. Likewise, $ \Omega ^+ (\mathbb{R}^2 ) $
  can be identified with $ \mathrm{Cliff} _{ 0,2} ^+ (\mathbb{R} )
  $-valued (i.e., complex-valued) functions on
  $\mathbb{R}^2$. Finally, we note that the Hodge--Dirac operator
  corresponds to the usual Dirac operator $ e _1 \partial _1 + e
  _2 \partial _2 $ for $ \mathrm{Cliff} _{ 0,2} (\mathbb{R} ) $-valued
  functions, since
  \begin{multline*}
    ( e _1 \partial _1 + e _2 \partial _2 ) ( u _0 + u _1 e _1 + u _2
    e _2 + u _{ 1 2 } e _1 e _2 )
    = - ( \partial _1 u _1 + \partial _2 u _2 ) \\
    + ( \partial _1 u _0 + \partial _2 u _{ 1 2 } ) e _1 + ( \partial
    _2 u _0 - \partial _1 u _{ 1 2 } ) e _2 + ( \partial _1 u _2
    - \partial _2 u _1 ) e _1 e _2 ,
  \end{multline*}
  which agrees with the previous expression for $ \mathrm{D} u $.
\end{example}

\begin{example}[$\mathbb{R}^3$ and vector calculus]
  A general element $ u \in \Omega ( \mathbb{R}^3 ) $ has the form
  \begin{multline*}
    u = u _0 + u _1 \,\mathrm{d}x _1 + u _2 \,\mathrm{d}x _2 + u _3
    \,\mathrm{d}x _3 \\
    + u _{ 2 3 } \,\mathrm{d}x _2 \wedge \mathrm{d}x _3 + u _{ 3 1 }
    \,\mathrm{d}x _3 \wedge \mathrm{d}x _1 + u _{ 1 2 }
    \,\mathrm{d}x _1 \wedge \mathrm{d}x _2 \\
    + u _{ 1 2 3 } \,\mathrm{d}x _1 \wedge \mathrm{d}x _2 \wedge
    \mathrm{d}x _3 .
  \end{multline*} 
  Defining the odd vector field $ \mathbf{u} ^- = ( u _1 , u _2 , u _3
  ) $ and even vector field $ \mathbf{u} ^+ = ( u _{ 2 3 } , u _{ 3 1
  } , u _{ 1 2 } ) $, the element $u$ can be written in the simpler
  form
  \begin{equation*}
    u = u _0 + \mathbf{u} ^- \cdot \mathrm{d} \mathbf{l} + \mathbf{u} ^+
    \cdot \mathrm{d} \mathbf{S} + u _{ 1 2 3 } \,\mathrm{d}V .
  \end{equation*}
  This notation is evocative of the fact that $1$-, $2$-, and
  $3$-forms correspond, respectively, to line, surface, and volume
  integrals in vector calculus. The exterior derivative and its dual
  are then given by
  \begin{align*}
    \mathrm{d} u &= \operatorname{grad} u _0 \cdot \mathrm{d}
    \mathbf{l} + \operatorname{curl} \mathbf{u} ^- \cdot \mathrm{d}
    \mathbf{S} + \operatorname{div} \mathbf{u} ^+ \,\mathrm{d} V ,\\
    \mathrm{d} ^\ast u &= - \operatorname{div} \mathbf{u} ^-+
    \operatorname{curl} \mathbf{u} ^+ \cdot \mathrm{d} \mathbf{l} -
    \operatorname{grad} u _{ 1 2 3 } \cdot \mathrm{d} \mathbf{S} ,
  \end{align*} 
  so the Hodge--Dirac operator $ \mathrm{D} = \mathrm{d} + \mathrm{d}
  ^\ast $ is defined by
  \begin{equation*}
    \mathrm{D} u = - \operatorname{div} \mathbf{u} ^- + (
    \operatorname{grad} u _0 + \operatorname{curl} \mathbf{u} ^+ )
    \cdot \mathrm{d} \mathbf{l} + ( \operatorname{curl} \mathbf{u} ^-
    - \operatorname{grad} u _{ 1 2 3 } ) \cdot \mathrm{d} \mathbf{S} +
    \operatorname{div} \mathbf{u} ^+ \,\mathrm{d}V.
  \end{equation*}
  In particular, restricting to the even part $ u ^+ = u _0 +
  \mathbf{u} ^+ \cdot \mathrm{d} \mathbf{S} $ yields
  \begin{equation*}
    \mathrm{D} u ^+ = ( \operatorname{grad} u _0 + \operatorname{curl}
    \mathbf{u} ^+) \cdot \mathrm{d} \mathbf{l} + \operatorname{div}
    \mathbf{u} ^+ \,\mathrm{d}V .
  \end{equation*} 
  Identifying $ u ^+ $ with the quaternion-valued function $ f = u _0
  + u _{ 2 3 } i + u _{ 3 1 } j + u _{ 1 2 } k $, it follows that $
  \mathrm{D} u ^+ = 0 $ if and only if
  \begin{equation*}
    \operatorname{grad} u _0 + \operatorname{curl} \mathbf{u} ^+ = 0 ,
    \qquad \operatorname{div}
    \mathbf{u} ^+ = 0 ,
  \end{equation*} 
  i.e., $f$ is a solution to the so-called \emph{Moisil--Th\'eodorescu
    system} \citep{MoTh1931}, a Cauchy--Riemann-type system arising in
  quaternionic analysis.

  In the language of Clifford analysis, the algebra $ \mathrm{Cliff}
  _{ 0, 3 } (\mathbb{R} ) $ on $\mathbb{R}^3$ has elements of the form
  \begin{equation*}
    a = a _0 + a _1 e _1 + a _2 e _2 + a _3 e _3 + a _{ 2 3 } e _2 e _3 + a _{ 3 1
    } e _3 e _1 + a _{ 1 2 } e _1 e _2 + a _{ 1 2 3 } e _1 e _2 e _3  ,
  \end{equation*} 
  subject to the algebraic relations $ e _1 ^2 = e _2 ^2 = e _3 ^2 = -
  1 $, along with the anti-commutativity relations $ e _1 e _2 = - e
  _2 e _1 $, $ e _1 e _3 = - e _3 e _1 $, and $ e _2 e _3 = - e _3 e
  _2 $. This is isomorphic to the algebra of so-called
  \emph{split-biquaternions} $ \mathbb{H} \oplus \mathbb{H} $, while
  the even subalgebra $ \mathrm{Cliff} _{ 0, 3 } ^+ (\mathbb{R} ) $ is
  isomorphic to $ \mathbb{H} $. As in the previous example, there is a
  bijective correspondence between $ \mathrm{Cliff} _{ 0, 3 }
  (\mathbb{R} ) $-valued (resp., $ \mathrm{Cliff} _{ 0, 3 } ^+
  (\mathbb{R} ) $-valued) functions and $ \Omega ( \mathbb{R}^3 ) $
  (resp., $ \Omega ^+ ( \mathbb{R}^3 ) $), whereby the Dirac operator
  corresponds to the Hodge--Dirac operator.
\end{example}

\section{Hodge--Dirac theory for nilpotent operators}
\label{sec:nilpotent}

In this section, we develop an abstract version of Hodge theory, where
the exterior derivative is replaced by an arbitrary nilpotent operator
on a Hilbert space. This abstraction has appeared earlier, notably in
the papers \citet{AxMc2004,AxKeMc2006}. After introducing this basic
machinery, we study an abstract Hodge--Dirac operator and its
associated variational problem, for which we prove a well-posedness
result.

\subsection{Nilpotent operators and the abstract Hodge decomposition}

\begin{definition}
  \label{def:nilpotent}
  Let $ \mathrm{d} $ be a closed, densely-defined linear operator on a
  (real) Hilbert space $W$, with domain $ \mathcal{D} (\mathrm{d}) $,
  range $ \mathcal{R} (\mathrm{d}) $, and kernel $ \mathcal{N}
  (\mathrm{d}) $.  Then we say that $ \mathrm{d} $ is
  \begin{enumerate}[(i)]
  \item \emph{nilpotent} if $ \mathcal{R}
    (\mathrm{d}) \subset \mathcal{N} (\mathrm{d}) $, so that $
    \mathrm{d} ^2 = 0 $;
  \item \emph{closed-nilpotent} if, in addition to (i), $
    \mathcal{R} (\mathrm{d}) $ is closed;
  \item \emph{Fredholm-nilpotent} if, in addition to (i)--(ii), $
    \mathcal{R} (\mathrm{d}) $ has finite codimension in $ \mathcal{N}
    (\mathrm{d}) $;
  \item \emph{diffuse Fredholm-nilpotent} if, in addition to
    (i)--(iii), $ \mathcal{D} (\mathrm{d}) \cap \mathcal{N}
    (\mathrm{d}) ^\perp $ is compact in $ \mathcal{N} (\mathrm{d})
    ^\perp $.
  \end{enumerate}
\end{definition}

Given a nilpotent operator $ \mathrm{d} $ on a Hilbert space $W$,
define the subspaces
\begin{equation*}
  \mathfrak{Z}  = \mathcal{N} ( \mathrm{d} ) , \qquad
  \mathfrak{B}  = \mathcal{R} (\mathrm{d}) \subset
  \mathfrak{Z}  , \qquad \mathfrak{H} = \mathfrak{Z}  \cap
  \mathfrak{B}  ^{\perp _W }.
\end{equation*} 
The notation for these spaces originates from homology theory: $
\mathfrak{B} $ stands for \emph{boundary}, $ \mathfrak{Z} $ stands for
\emph{cycle} (German: \emph{Zyklus}), and $ \mathfrak{H} $ stands
for \emph{harmonic}.  It follows immediately that $W$ has the
orthogonal decomposition
\begin{equation*}
  W = \mathfrak{Z}  \oplus \mathfrak{Z}  ^{ \perp _W } = \mathfrak{Z}
  \cap  ( \overline{ \mathfrak{B}  } \oplus \mathfrak{B}  ^{
      \perp _W } ) \oplus \mathfrak{Z}  ^{ \perp _W } =
  \overline{ \mathfrak{B}  } \oplus \mathfrak{H}  \oplus \mathfrak{Z}
  ^{ \perp _W } ,
\end{equation*} 
which we call the \emph{abstract Hodge decomposition} of $W$.  The
adjoint $ \mathrm{d} ^\ast $ is also a nilpotent operator, so we can
define the corresponding subspaces $ \mathfrak{Z} ^\ast = \mathcal{N}
( \mathrm{d} ^\ast ) $ and $ \mathfrak{B} ^\ast = \mathcal{R}
(\mathrm{d} ^\ast ) \subset \mathfrak{Z} ^\ast $.  However, since $
\mathrm{d} $ and $ \mathrm{d} ^\ast $ are adjoints, it is possible to
write
\begin{equation*}
  \mathfrak{Z}  ^\ast = \mathfrak{B}  ^{ \perp _W } , \qquad
  \overline{ \mathfrak{B}  ^\ast } = \mathfrak{Z}  ^{ \perp _W } ,
  \qquad  \mathfrak{H}  = \mathfrak{Z}  \cap \mathfrak{Z}  ^\ast .
\end{equation*} 
Hence, the Hodge decomposition has the alternative form,
\begin{equation*}
  W = \overline{ \mathfrak{B}  } \oplus \mathfrak{H}  \oplus
  \overline{   \mathfrak{B}  ^\ast } .
\end{equation*} 
When $ \mathrm{d} $ is closed-nilpotent, so is $ \mathrm{d} ^\ast $,
and so $ \mathfrak{B} $ and $ \mathfrak{B} ^\ast $ are closed
subspaces.  In this case, the Hodge decomposition becomes
\begin{equation*}
  W = \mathfrak{B}  \oplus \mathfrak{H}  \oplus \mathfrak{Z}  ^{ \perp
    _W } = \mathfrak{B}  \oplus \mathfrak{H}  \oplus \mathfrak{B}
  ^\ast .
\end{equation*} 
Finally, note that $ \mathrm{d} $ is (diffuse) Fredholm-nilpotent if
and only if $ \mathrm{d} ^\ast $ is.

The operator $ \mathrm{d} $ is generally unbounded on $W$.  However,
we can equip the dense domain $ V = \mathcal{D} (\mathrm{d})
\subset W $ with its own Hilbert space structure, so that $ \mathrm{d}
$ is a bounded operator on $V$.  Denoting the inner
product on $W$ by $ \langle \cdot, \cdot \rangle $, let $V$
be endowed with the graph inner product,
\begin{equation*}
  \langle u, v \rangle _V = \langle u, v \rangle + \langle \mathrm{d}
  u, \mathrm{d} v \rangle, \qquad \forall u , v \in V . 
\end{equation*} 
Since $ \mathrm{d} $ is a closed operator, its graph is closed, and
hence $V$ is complete with respect to the norm $ \lVert \cdot \rVert
_V $ induced by the inner product.  The operator $ \mathrm{d} $ is
bounded (in fact, nonexpansive) with respect to the $V$-norm, since
\begin{equation*}
  \lVert \mathrm{d} v \rVert ^2 _V = \lVert \mathrm{d} v \rVert ^2
  \leq \lVert v \rVert ^2 + \lVert \mathrm{d} v \rVert ^2 = \lVert v
  \rVert _V ^2 .
\end{equation*} 
Since $ \mathfrak{B} $ and $ \mathfrak{Z} $ are both in $V$, we again
obtain an abstract Hodge decomposition,
\begin{equation*}
  V = \overline{ \mathfrak{B}  } \oplus \mathfrak{H}  \oplus
  \mathfrak{Z}  ^\perp,
\end{equation*} 
where $ \mathfrak{Z} ^\perp = \mathfrak{Z} ^{ \perp _W } \cap
V $.  Finally, $ \mathrm{d} $ is closed- or Fredholm-nilpotent on $V$
if and only if it is on $W$, and in this case we have
\begin{equation*}
  V = \mathfrak{B}  \oplus \mathfrak{H}  \oplus \mathfrak{Z}  ^\perp ,
\end{equation*} 
as before.  Given $ v \in V $, we denote the components of its Hodge
decomposition by $ v = v _{ \mathfrak{B} } + v _{ \mathfrak{H} } + v
_{ \perp } $.

The Poincar\'e inequality is one of the most important corollaries of
the Hodge decomposition, from an analytical standpoint.  This was
demonstrated in \citet{ArFaWi2010} in the abstract setting of closed
Hilbert complexes; in the special case of the de~Rham complex, one
obtains the classical Poincar\'e inequality.  By the same argument,
one also obtains an abstract Poincar\'e inequality for
closed-nilpotent operators.

\begin{lemma}[Poincar\'e inequality]
  \label{lem:poincare}
  If $ \mathrm{d} $ is a closed-nilpotent operator, then there exists
  a constant $ c _P \geq 1 $, called the Poincar\'e constant, such
  that
  \begin{equation*}
    \lVert v \rVert _V \leq c _P \lVert \mathrm{d} v \rVert , \qquad
    \forall v \in \mathfrak{Z}  ^\perp .
  \end{equation*} 
\end{lemma}

\begin{proof}
  The linear map $ \mathrm{d} $ restricts to a bounded bijection
  between $ \mathfrak{Z} ^\perp $ and $ \mathfrak{B} $, both of which
  are closed subspaces of $V$.  Hence, Banach's bounded inverse
  theorem (a standard corollary of the open mapping theorem) implies
  that $ \mathrm{d} \rvert _{ \mathfrak{Z} ^\perp } $ has a bounded
  inverse, which proves the result. (The bound $c_P \geq 1$ is a
  result of the nonexpansiveness of $\mathrm{d}$ with respect to the
  $V$ norm.)
\end{proof}

\subsection{The abstract Hodge--Dirac operator}

Having defined abstract versions of the exterior derivative and Hodge
decomposition, we are now prepared to define the abstract Hodge--Dirac
operator.

\begin{definition}
  Given a Hilbert space $W$ with a nilpotent operator $ \mathrm{d} $
  and adjoint $ \mathrm{d} ^\ast $, the \emph{abstract Hodge--Dirac
    operator} is $ \mathrm{D} = \mathrm{d} + \mathrm{d} ^\ast $.
\end{definition}

The Hodge--Dirac operator inherits several of the properties of $
\mathrm{d} $ and $ \mathrm{d} ^\ast $.  Like those operators, it is
closed and densely-defined, and its domain, range, and kernel are
given by
\begin{equation*}
  \mathcal{D} (\mathrm{D}) = \mathcal{D} (\mathrm{d}) \cap \mathcal{D}
  (\mathrm{d} ^\ast ), \qquad \mathcal{R} (\mathrm{D}) = \mathfrak{B}
  \oplus \mathfrak{B}  ^\ast , \qquad \mathcal{N} (\mathrm{D}) =
  \mathfrak{Z}  \cap \mathfrak{Z}  ^\ast = \mathfrak{H}  .
\end{equation*} 
Moreover, $ \mathrm{D} $ has closed range if and only if $ \mathrm{d}
$ is closed-nilpotent, and is (diffuse) Fredholm if and only if $
\mathrm{d} $ is (diffuse) Fredholm-nilpotent (cf.~\citet[Propositions
3.5 and 3.11]{AxMc2004}).  Unlike $ \mathrm{d} $ and $ \mathrm{d}
^\ast $, though, $ \mathrm{D} $ is self-adjoint and is \emph{not}
nilpotent; in fact, its square is
\begin{equation*}
  \mathrm{D} ^2 = \mathrm{d} \mathrm{d} ^\ast + \mathrm{d} ^\ast
  \mathrm{d} = L ,
\end{equation*} 
which is called the \emph{abstract Hodge--Laplace operator}.

The Hodge decomposition implies that $ W = \overline{ \mathcal{R}
  (\mathrm{D}) } \oplus \mathcal{N} (\mathrm{D}) $.  In particular,
when $ \mathrm{d} $ is closed-nilpotent, $ \mathcal{R} (\mathrm{D}) $
is closed, so the Hodge decomposition is $ W = \mathcal{R}
(\mathrm{D}) \oplus \mathcal{N} (\mathrm{D}) $.  In this latter case,
the Hodge decomposition is simply an expression of the closed range
theorem for the self-adjoint operator $ \mathrm{D} $.  This
decomposition makes it natural to pose the following problem: Given $
f \in W $, find $ (u,p) \in \bigl( \mathcal{D} (D) \cap \mathcal{N}
(\mathrm{D}) ^\perp \bigr) \oplus \mathcal{N} (\mathrm{D}) $
satisfying
\begin{equation*}
  \mathrm{D} u + p = f .
\end{equation*} 
The solution to this problem gives the Hodge decomposition $ f =
\mathrm{d} u + \mathrm{d} ^\ast u + p $.

We now consider the associated variational problem: Find $ ( u, p )
\in V \times \mathfrak{H} $ such that
\begin{equation}
  \label{eqn:hodgeDirac}
\begin{alignedat}{2}
  \langle \mathrm{d} u, v \rangle + \langle u, \mathrm{d} v \rangle +
  \langle p, v \rangle &= \langle f, v \rangle , \quad &\forall v &\in
  V ,\\
  \langle u, q \rangle &= 0 , \quad &\forall q &\in \mathfrak{H} .
\end{alignedat}
\end{equation}
If we define the bilinear form $ B \colon ( V \times \mathfrak{H} )
\times ( V \times \mathfrak{H} ) \rightarrow \mathbb{R} $ to be
\begin{equation*}
  B ( u, p; v, q ) = \langle \mathrm{d} u, v \rangle + \langle u ,
  \mathrm{d} v \rangle + \langle p, v \rangle + \langle u, q \rangle ,
\end{equation*} 
then the variational problem can be rewritten as: Find $ ( u, p ) \in
V \times \mathfrak{H} $ such that
\begin{equation*}
  B ( u, p; v , q ) = \langle f, v \rangle ,
  \quad \forall ( v, q ) \in V \times \mathfrak{H}  .
\end{equation*} 
Note that $B$ is bounded (by a straightforward application of the
Cauchy--Schwarz inequality) and symmetric. Therefore, in order to
establish the well-posedness of this problem, it suffices to prove the
``inf-sup condition''
\begin{equation*}
  \gamma = \infsup{ (0,0) \neq ( u, p ) \in V \times \mathfrak{H}  }{
    (0,0) \neq ( v, q ) \in V \times \mathfrak{H}  } \frac{
    \bigl\lvert B ( u, p; v , q ) \bigr\rvert }{ \bigl\lVert ( u, p )
    \bigr\rVert _{ V \times \mathfrak{H}  } \bigl\lVert ( v, q )
    \bigr\rVert _{ V \times \mathfrak{H}  } } > 0 ,
\end{equation*} 
which is implied by the following theorem. (Compare \citet[Theorem
3.2]{ArFaWi2010}.)

\begin{theorem}
  \label{thm:infSup}
  Suppose $ \mathrm{d} $ is a closed-nilpotent operator on a Hilbert
  space $W$ with dense domain $ V \subset W $.  Then there exists a
  constant $ \gamma > 0 $, depending only on the Poincar\'e constant $
  c _P $, such that for all nonzero $ ( u, p ) \in V \times
  \mathfrak{H} $, there exists a nonzero $ ( v, q ) \in V \times
  \mathfrak{H} $ satisfying
  \begin{equation*}
    B ( u, p ; v, q ) \geq \gamma \bigl(  \lVert u
    \rVert _V + \lVert p \rVert \bigr) \bigl( 
    \lVert v \rVert _V + \lVert q \rVert \bigr) .
  \end{equation*} 
\end{theorem}

\begin{notation}
  In the following proof, and in the remainder of this paper, we
  follow the common practice of letting $C$ denote an unspecified,
  positive constant, whose value may differ with each occurrence (even
  within the same proof).
\end{notation}

\begin{proof}
  Take the test functions
  \begin{equation*}
    v = \rho + p + \mathrm{d} u , \qquad q = u _{ \mathfrak{H}  } ,
  \end{equation*} 
  where $ \rho \in \mathfrak{Z}  ^\perp $ is the unique element such
  that $ \mathrm{d} \rho = u _{ \mathfrak{B}  } $.  Using the
  Poincar\'e inequality and the orthogonality of the Hodge
  decomposition, observe that
  \begin{align*}
    \lVert v \rVert _V + \lVert q \rVert &\leq \lVert \rho \rVert _V +
    \lVert p \rVert + \lVert \mathrm{d} u \rVert + \lVert u _{
      \mathfrak{H}  } \rVert \\
    &\leq c _P \lVert u _{ \mathfrak{B} } \rVert + \lVert u _{
      \mathfrak{H} } \rVert + \lVert
    \mathrm{d} u \rVert + \lVert p \rVert \\
    &\leq C \bigl( \lVert u \rVert _V + \lVert p \rVert \bigr) .
  \end{align*}
  Next, substituting these test functions into the bilinear form,
  \begin{align*}
    B ( u, p; v, q ) &= \lVert \mathrm{d} u \rVert ^2 + \langle u , u
    _{ \mathfrak{B} } \rangle + \lVert p \rVert ^2 + \langle u, u
    _{ \mathfrak{H}  } \rangle \\
    &= \lVert \mathrm{d} u \rVert ^2 + \lVert u _{ \mathfrak{B} }
    \rVert ^2 + \lVert u _{ \mathfrak{H} } \rVert ^2 + \lVert p \rVert
    ^2 \\
    &= \frac{1}{2} \lVert \mathrm{d} u \rVert ^2 + \frac{1}{2} \lVert
    \mathrm{d} u _{ \perp } \rVert ^2 + \lVert u _{ \mathfrak{B} }
    \rVert ^2 + \lVert u _{ \mathfrak{H} } \rVert ^2 + \lVert p \rVert
    ^2
    \\
    &\geq \frac{1}{2} \lVert \mathrm{d} u \rVert ^2 + \frac{1}{2 c _P
      ^2 } \lVert u _{ \perp } \rVert ^2 + \lVert u _{ \mathfrak{B} }
    \rVert ^2 + \lVert u _{ \mathfrak{H} } \rVert ^2 + \lVert p \rVert
    ^2
    \\
    &\geq \frac{ 1 }{ 2 c _P ^2 } \bigl( \lVert u \rVert ^2 _V +
    \lVert p \rVert ^2 \bigr),
  \end{align*}
  where the last inequality follows from $ c_P \geq 1 $ and the Hodge
  decomposition.  Combining this with the previous inequality, we
  therefore obtain
  \begin{equation*}
    B ( u, p; v, q ) \geq \gamma \bigl( \lVert u \rVert _V + \lVert p
    \rVert \bigr) \bigl( \lVert v \rVert _V + \lVert q \rVert \bigr),
  \end{equation*} 
  as claimed. 
\end{proof}

\begin{corollary}
  \label{cor:hodgeDiracWellPosed}
  The variational problem for the abstract Hodge--Dirac operator is
  well-posed.  That is, there exists a constant $ c $, depending
  only on the Poincar\'e constant $ c _P $, such that for all $ f \in
  W $, the problem \eqref{eqn:hodgeDirac} has a unique solution $
  ( u, p ) \in V \times \mathfrak{H} $, which satisfies the
  estimate
  \begin{equation*}
    \lVert u \rVert _V + \lVert p \rVert \leq c \lVert f \rVert .
  \end{equation*}
\end{corollary}

\begin{proof}
  This follows directly from the inf-sup condition, cf.~\citet{Babuska1971}.
\end{proof}

\section{Numerical stability and convergence of a discrete problem}
\label{sec:discrete}

In this section, we discuss the approximation of the Hodge--Dirac
variational problem \eqref{eqn:hodgeDirac} on a closed subspace $ V _h
\subset V $. We will often refer to this as the \emph{discrete
  Hodge--Dirac problem}, since $ V _h $ is typically a
finite-dimensional subspace obtained by some discretization process,
e.g., finite-element discretization.  First, in
\autoref{sec:projection}, we discuss some additional structure that
must be assumed---crucially, as in \citet{ArFaWi2010}, we require the
existence of a bounded projection, which along with the inclusion map
$ V _h \hookrightarrow V $ must commute with the differentials---and
the consequences of this additional structure. Next, in
\autoref{sec:estimates}, we introduce the discrete problem and prove
stability and convergence estimates.  Finally, in
\autoref{sec:improved}, we give improved estimates for the case where
$ \mathrm{d} $ is not merely closed-nilpotent but diffuse
Fredholm-nilpotent.

\subsection{Approximation by a subspace with a bounded commuting
  projection}
\label{sec:projection}

Let $ V _h \subset V $ be a closed (e.g., finite-dimensional) subspace
of $V$, such that $ \mathrm{d} V _h \subset V _h $.  If $ \mathrm{d} $
is closed-nilpotent on $V$, then the restriction $ \mathrm{d} _h =
\mathrm{d} \rvert _{ V _h } $ is closed-nilpotent on $ V _h $, and
this induces an abstract Hodge decomposition $ V _h = \mathfrak{B} _h
\oplus \mathfrak{H} _h \oplus \mathfrak{Z} _h ^\perp $. Note that,
although $ \mathrm{d} _h $ is the restriction of $ \mathrm{d} $ to $ V
_h $, its adjoint $ \mathrm{d} _h ^\ast $ with respect to the
$W$-inner product is generally \emph{not} the restriction of $
\mathrm{d} ^\ast $; consequently, we have $ \mathfrak{B} _h \subset
\mathfrak{B} $ and $ \mathfrak{Z} _h \subset \mathfrak{Z} $, but
generally $ \mathfrak{H} _h \not\subset \mathfrak{H} $ and $
\mathfrak{Z} _h ^\perp \not\subset \mathfrak{Z} ^\perp $.

We now make one additional assumption: suppose also that there exists
a bounded projection $ \pi _h \in \mathcal{L} ( V , V _h ) $ such that
$ \pi _h \mathrm{d} v = \mathrm{d} \pi _h v $ for all $ v \in V $. (By
``projection,'' we mean only that $ \pi _h $ is idempotent and
surjective onto $ V _h $, not that it is an orthogonal projection.)
It is nontrivial to show that such projections exist, so their
explicit construction for the Hodge--de~Rham complex
(cf.~\citet{ChWi2008,FaWi2013}) was a major technical advance in
finite element exterior calculus.  The importance of this assumption
is that it allows us to control the Poincar\'e constant of $ V _h $ in
terms of $ c _P $ and $ \lVert \pi _h \rVert $, as shown in the
following lemma (essentially similar to \citet[Theorem
3.6]{ArFaWi2010}).

\begin{lemma}
  \label{lem:discretePoincare}
  Let $ \mathrm{d} $ be a closed-nilpotent operator on a Hilbert space
  $W$ with dense domain $ V \subset W $, and suppose that $ V _h
  \subset V $ is a closed subspace with a bounded commuting projection
  $ \pi _h \in \mathcal{L} ( V, V _h ) $ such that $ \pi _h \mathrm{d}
  = \mathrm{d} \pi _h $. Then
  \begin{equation*}
    \lVert v _h \rVert _V \leq c _P \lVert \pi _h \rVert \lVert
    \mathrm{d} v _h \rVert , \qquad \forall v _h \in \mathfrak{Z}  _h ^\perp .
  \end{equation*} 
  In other words, the Poincar\'e constant of $ V _h $ is bounded by $
  c _P \lVert \pi _h \rVert $.
\end{lemma}

\begin{proof}
  Given $ v _h \in \mathfrak{Z} _h ^\perp $, let $ z \in \mathfrak{Z}
  ^\perp $ be the element satisfying $ \mathrm{d} z = \mathrm{d} v _h
  $. This $z$ exists and is unique, since $ \mathrm{d} \rvert _{
    \mathfrak{Z} ^\perp } $ is a bijection between $ \mathfrak{Z}
  ^\perp $ and $ \mathfrak{B} \supset \mathfrak{B} _h \ni \mathrm{d} v
  _h $. Moreover, by \autoref{lem:poincare}, we have
  \begin{equation*}
    \lVert z \rVert _V \leq c _P \lVert \mathrm{d} z \rVert = c _P
    \lVert \mathrm{d} v _h \rVert .
  \end{equation*} 
  Thus, it suffices to show $ \lVert v _h \rVert _V \leq \lVert \pi _h
  \rVert \lVert z \rVert _V $.  Observe that
  \begin{equation*}
    \mathrm{d} v _h = \pi _h \mathrm{d} v _h = \pi _h \mathrm{d} z =
    \mathrm{d} \pi _h z ,
  \end{equation*} 
  which implies that $ \mathrm{d} ( v _h - \pi _h z ) = 0
  $. Hence, $ v _h - \pi _h z \in \mathfrak{Z} _h \perp v _h $, so
  \begin{equation*}
    \lVert v _h \rVert _V ^2 = 
    \langle v _h , v _h - \pi _h z \rangle _V + \langle v _h , \pi _h
    z \rangle _V = \langle v _h , \pi _h z \rangle _V \leq \lVert v _h
    \rVert _V \lVert \pi _h \rVert \lVert z \rVert _V .
  \end{equation*} 
  Finally, dividing through by $ \lVert v _h \rVert _V $ completes the proof. 
\end{proof}

Unlike with orthogonal projection, generally $ \pi _h v $ is not the
best approximation to $v$ in $ V _h $. However, it is nearly as good:
since $ \pi _h v _h = v _h $ for all $ v _h \in V _h $, we have
\begin{equation*}
  \lVert v - \pi _h v \rVert _V = \bigl\lVert ( I - \pi _h ) v
  \bigr\rVert _V = \bigl\lVert ( I - \pi _h ) ( v - v _h ) \bigr\rVert _V \leq C \lVert v - v _h \rVert _V ,
\end{equation*} 
and therefore
\begin{equation*}
  \lVert v - \pi _h v \rVert _V \leq C \inf _{ v _h \in V _h } \lVert
  v - v _h \rVert _V .
\end{equation*} 
In other words, the approximation error differs from the optimum by at
most a constant factor, a property known as \emph{quasi-optimality}.

\subsection{Stability and convergence of the discrete problem}
\label{sec:estimates}

Since $ \mathrm{d} _h $ is a closed-nilpotent operator on $ V _h $, we
may define a discrete Hodge--Dirac operator $ \mathrm{D} _h =
\mathrm{d} _h + \mathrm{d} _h ^\ast $. (Note that, since $ \mathrm{d}
_h ^\ast $ is generally not the restriction of $ \mathrm{d} ^\ast $ to
$ V _h $, neither is $ \mathrm{D} _h $ simply the restriction of $
\mathrm{D} $.)  The discrete version of the variational problem
\eqref{eqn:hodgeDirac} is then: Find $ ( u _h , p _h ) \in V _h \times
\mathfrak{H} _h $ satisfying
\begin{equation}
  \label{eqn:discreteHodgeDirac}
\begin{alignedat}{2}
  \langle \mathrm{d} u _h , v _h \rangle + \langle u _h , \mathrm{d} v
  _h \rangle + \langle p _h , v _h \rangle &= \langle f, v _h \rangle
  , \quad &\forall v _h &\in V _h ,\\
  \langle u _h , q _h \rangle &= 0 , \quad &\forall q _h &\in
  \mathfrak{H} _h .
\end{alignedat}
\end{equation}
Once again, since generally $ \mathfrak{H} _h \not\subset \mathfrak{H}
$, it follows that $ V _h \times \mathfrak{H} _h \not\subset V \times
\mathfrak{H} $. Hence, \eqref{eqn:discreteHodgeDirac} is not simply
a Galerkin discretization of the continuous variational problem
\eqref{eqn:hodgeDirac}; rather, there is a ``variational crime'' that
will need to be accounted for in the subsequent numerical analysis.

The following theorem gives a discrete inf-sup condition, thereby
showing the stability of the discretization and the well-posedness of
\eqref{eqn:discreteHodgeDirac}.

\begin{theorem}
  \label{eqn:discreteInfSup}
  Let $W$ be a Hilbert space with a closed-nilpotent operator $
  \mathrm{d} $ defined on the dense domain $ V \subset W $.  Suppose
  that $ V _h \subset V $ is a closed subspace satisfying $ \mathrm{d}
  V _h \subset V _h $ and equipped with a bounded projection $ \pi _h
  \colon V \rightarrow V _h $ such that $ \pi _h \mathrm{d} =
  \mathrm{d} \pi _h $.  Then there exists a constant $ \gamma _h > 0
  $, depending only on the Poincar\'e constant $ c _P $ and on the
  norm of $ \pi _h $, such that for all nonzero $ ( u _h , p _h ) \in
  V _h \times \mathfrak{H} _h $, there exists a nonzero $ ( v _h , q
  _h ) \in V _h \times \mathfrak{H} _h $ satisfying
  \begin{equation*}
    B ( u _h , p _h ; v _h , q _h ) \geq \gamma _h \bigl( \lVert u _h
    \rVert _V + \lVert p _h \rVert \bigr) \bigl( \lVert v _h \rVert _V + \lVert
    q _h \rVert \bigr)  .
  \end{equation*} 
\end{theorem}

\begin{proof}
  This follows immediately from \autoref{thm:infSup} and
  \autoref{lem:discretePoincare}.
\end{proof}

\begin{corollary}
  \label{cor:discreteHodgeDiracWellPosed}
  The discrete variational problem for the Hodge--Dirac operator is
  well-posed.  That is, there exists a constant $ c _h $,
  depending only on the Poincar\'e constant $ c _P $ and on the norm
  of $ \pi _h $, such that for all $ f \in W $, the problem
  \eqref{eqn:discreteHodgeDirac} has a unique solution $ ( u _h , p _h
  ) \in V _h \times \mathfrak{H} _h $, which satisfies
  \begin{equation*}
    \lVert u _h \rVert _V + \lVert p _h \rVert \leq c _h \lVert f \rVert .
  \end{equation*} 
\end{corollary}

The next theorem provides a quasi-optimal \textit{a priori} error
estimate for the approximation of solutions to \eqref{eqn:hodgeDirac}
by those to \eqref{eqn:discreteHodgeDirac}. (Compare \citet[Theorem
3.9]{ArFaWi2010}.)

\begin{notation}
  In the statement of the theorem, $ P _{ \mathfrak{B} } $ and $ P _{
    \mathfrak{H} } $ denote the $W$-orthogonal projections onto $
  \mathfrak{B} $ and $ \mathfrak{H} $, respectively. We will use
  similar notation, throughout the remainder of the paper, to denote
  $W$-orthogonal projection onto these and other closed subspaces.
\end{notation}

\begin{theorem}
  \label{thm:errorEstimate}
  Let $W$ be a Hilbert space with a closed-nilpotent operator $
  \mathrm{d} $ defined on the dense domain $ V \subset W $.  Suppose
  that $ V _h \subset V $ is a family of closed subspaces,
  parametrized by $ h $, satisfying $ \mathrm{d} V _h \subset V _h $.
  Suppose also that these subspaces are equipped with projections $
  \pi _h \colon V \rightarrow V _h $, bounded uniformly in $h$, such
  that $ \pi _h \mathrm{d} = \mathrm{d} \pi _h $.  If $ ( u, p ) \in V
  \times \mathfrak{H} $ solves \eqref{eqn:hodgeDirac} and $ ( u _h , p
  _h ) \in V _h \times \mathfrak{H} _h $ solves
  \eqref{eqn:discreteHodgeDirac}, then we have the error estimate
  \begin{equation*}
    \lVert u - u _h \rVert _V + \lVert p - p _h \rVert 
    \leq C \biggl( \inf _{ v \in V _h } \lVert u - v \rVert _V + \inf
    _{ q \in V _h } \lVert p - q \rVert _V + \mu \inf _{ v \in V _h }
    \lVert P _{ \mathfrak{B} } u - v \rVert _V \biggr) ,
  \end{equation*}
  where $ \mu = \bigl\lVert ( I - \pi _h ) P _{ \mathfrak{H} }
  \bigr\rVert $.
\end{theorem}

\begin{proof}
  First, observe from the variational principles
  \eqref{eqn:hodgeDirac} and \eqref{eqn:discreteHodgeDirac} that
  \begin{equation*}
    B ( u , p; v _h , q _h ) = \langle f, v _h \rangle + \langle u , q
    _h \rangle = B ( u _h , p _h ; v _h , q _h ) + \langle u , q _h
    \rangle .
  \end{equation*} 
  Now, let $v$ and $q$ be the $V$-orthogonal projections of $u$ and
  $p$ onto $ V _h $ and $ \mathfrak{H} _h $, respectively. Then, using
  the previous observation and the boundedness of the bilinear form
  $B$, we have
  \begin{multline*}
    B ( u _h - v , p _h - q ; v _h , q _h ) = B ( u - v , p - q ; v _h
    , q _h ) - \langle u , q _h  \rangle \\
    \leq C \bigl( \lVert u - v \rVert _V + \lVert p - q \rVert +
    \lVert P _{ \mathfrak{H} _h } u \rVert \bigr) \bigl( \lVert v _h
    \rVert _V + \lVert q _h \rVert \bigr) .
  \end{multline*} 
  Therefore, applying the discrete inf-sup condition yields
  \begin{equation*} 
    \lVert u _h - v \rVert _V + \lVert p _h - q \rVert \leq C \bigl(
    \lVert u - v \rVert _V + \lVert p - q \rVert + \lVert P _{
      \mathfrak{H} _h } u \rVert \bigr).
  \end{equation*} 
  It now remains to estimate the terms $ \lVert p - q \rVert $ and $
  \lVert P _{ \mathfrak{H} _h } u \rVert $.

  For the former term, note that $ p \in \mathfrak{H} \perp
  \mathfrak{B} \supset \mathfrak{B} _h $, so 
  \begin{equation*}
    P _{ \mathfrak{Z} _h } p = P _{ \mathfrak{H} _h } p + P _{
      \mathfrak{B} _h } p = q + 0 = q .
  \end{equation*} 
  On the other hand, since $ p \in \mathfrak{H} \subset \mathfrak{Z}
  $, it follows that $ \pi _h p \in \mathfrak{Z} _h $. Hence,
  \begin{equation}
    \label{eqn:p-q}
    \lVert p - q \rVert = \bigl\lVert ( I - P _{ \mathfrak{Z}  _h } )
    p \bigr\rVert \leq \bigl\lVert ( I - \pi _h ) p \bigr\rVert \leq C
    \inf _{ q \in V _h } \lVert p - q \rVert _V ,
  \end{equation} 
  where we have used the optimality property of $ P _{ \mathfrak{Z} _h
  } $ and the quasi-optimality property of $ \pi _h $. (Compare
  \citet[Theorem 3.5]{ArFaWi2010}.)

  For the latter term, we have $ u \perp \mathfrak{H} $, so its Hodge
  decomposition can be written $ u = u _{ \mathfrak{B} } + u _{ \perp
  } $. However, $ u _{ \perp } \perp \mathfrak{Z} \supset \mathfrak{H}
  _h $, so $ P _{ \mathfrak{H} _h } u = P _{ \mathfrak{H} _h } u _{
    \mathfrak{B} } $. Furthermore, since $ \pi _h u _{ \mathfrak{B} }
  \in \mathfrak{B} _h \perp \mathfrak{H} _h $, we have
  \begin{equation*}
    P _{ \mathfrak{H}  _h } u = P _{ \mathfrak{H}  _h } u _{
      \mathfrak{B}  } = P _{ \mathfrak{H}  _h } ( u _{ \mathfrak{B}  }
    - \pi _h u _{ \mathfrak{B}  } ) = P _{ \mathfrak{H}  _h } ( I -
    \pi _h ) u _{ \mathfrak{B}  } .
  \end{equation*} 
  This implies $ \lVert P _{ \mathfrak{H} _h } u \rVert ^2 =
  \bigl\langle ( I - \pi _h ) u _{ \mathfrak{B} } , P _{ \mathfrak{H}
    _h } u \bigr\rangle $, and since $ ( I - \pi _h ) u _{
    \mathfrak{B} } \in \mathfrak{B} \perp \mathfrak{H} $,
  \begin{equation*}
    \lVert P _{ \mathfrak{H}  _h } u \rVert ^2 = \bigl\langle ( I -
    \pi _h ) u _{ \mathfrak{B} } , ( I - P _{ \mathfrak{H}  } ) P _{
      \mathfrak{H}  _h } u \bigr\rangle \leq \bigl\lVert ( I - \pi _h
    ) u _{ \mathfrak{B}  } \bigr\rVert \bigl\lVert ( I - P _{
      \mathfrak{H}  } ) P _{ \mathfrak{H}  _h } u \bigr\rVert .
  \end{equation*} 
  Next, since $ P _{ \mathfrak{H} _h } u \in \mathfrak{H} _h \subset
  \mathfrak{Z} $, we have $ ( I - P _{ \mathfrak{H} } ) P _{
    \mathfrak{H} _h } u \in \mathfrak{B} $. This implies that $ \pi _h
  ( I - P _{ \mathfrak{H} }) P _{ \mathfrak{H} _h } u \in \mathfrak{B}
  _h $ is orthogonal to both $ P _{ \mathfrak{H} _h } u $ and $ P _{
    \mathfrak{H} } P _{ \mathfrak{H} _h } u $, and hence to $ ( I - P
  _{ \mathfrak{H} } ) P _{ \mathfrak{H} _h } u $, so by the
  Pythagorean theorem,
  \begin{align*}
    \bigl\lVert ( I - P _{ \mathfrak{H} } ) P _{ \mathfrak{H} _h } u
    \bigr\rVert &\leq \bigl\lVert ( I - P _{ \mathfrak{H} } ) P _{
      \mathfrak{H} _h } u - \pi _h ( I - P _{
      \mathfrak{H}  } ) P _{ \mathfrak{H}  _h } u \bigr\rVert \\
    &= \bigl\lVert ( I - \pi _h ) P _{ \mathfrak{H} } P _{
      \mathfrak{H} _h } u \bigr\rVert \\
    &\leq \mu \lVert P _{ \mathfrak{H}  _h } u \rVert .
  \end{align*} 
  Finally, combining this with the estimate above for $ \lVert P _{
    \mathfrak{H} _h } u \rVert ^2 $, and using the quasi-optimality
  property of $ \pi _h $, we have
  \begin{equation*}
    \lVert P _{ \mathfrak{H}  _h } u \rVert \leq \mu \bigl\lVert ( I -
    \pi _h ) u _{\mathfrak{B}}  \bigr\rVert \leq C \mu \inf _{ v \in V
      _h } \lVert u _{ \mathfrak{B}  } - v \rVert _V .
  \end{equation*} 
  Altogether, it has now been shown that 
  \begin{equation*} 
    \lVert u _h - v \rVert _V + \lVert p _h - q \rVert \leq C \biggl(
    \inf _{ v \in V _h } \lVert u - v \rVert _V + \inf _{ q \in
      V  _h } \lVert p - q \rVert _V + \mu \inf _{ v \in V _h }
    \lVert P _{ \mathfrak{B}  }  u - v \rVert _V \biggr),
  \end{equation*} 
  so the result follows by an application of the triangle inequality.
\end{proof}

Finally, note that if the family of subspaces
$ \{ V _h \} _{ h > 0 } $ is \emph{pointwise approximating} in $V$, in
the sense that
\begin{equation}
\label{eqn:approx}
 \inf _{ v _h \in V _h } \lVert v - v _h \rVert _V \rightarrow 0
 \text{  as } h \rightarrow 0 \text{ for all }  v \in V ,
\end{equation}
then \autoref{thm:errorEstimate} immediately implies that
$ ( u _h , p _h ) \rightarrow ( u, p ) $ in $ V \times V $.

\subsection{Improved error estimates for diffuse Fredholm operators}
\label{sec:improved}

In this section, we obtain improved error estimates under the stronger
assumption that $ \mathrm{d} $ is not only closed-nilpotent, but is
diffuse Fredholm-nilpotent, cf.~\autoref{def:nilpotent}. The approach
is related to other improved estimates obtained using duality
techniques; these are known by various names, such as the
``Aubin--Nitsche trick'' and ``$ L ^2 $ lifting'' (cf. \citet[Theorem
3.2.4]{Ciarlet1978}).  The proofs given here owe a considerable debt
to \citet[Section~3.5]{ArFaWi2010}, whose techniques for the
Hodge--Laplace problem we have adapted to the Hodge--Dirac problem,
with some modifications.

Assume now that $ \mathrm{d} $ is a diffuse Fredholm-nilpotent
operator. It follows that $ \mathrm{D} $ is diffuse Fredholm, so the
solution operator $K$ on $W$, which takes $ f \mapsto u $, is
compact. Moreover, $ P _{ \mathfrak{H} } $ is also compact, since the
Fredholm property implies $ \dim \mathfrak{H} < \infty $. Finally, we
add the assumption that $ \pi _h $ is a bounded operator on $W$,
whereas previously, we had assumed only that it was bounded on
$V$.

Following \citet{ArFaWi2010}, we denote
$ \eta = \bigl\lVert ( I - \pi _h ) K \bigr\rVert $ and, as before,
$ \mu = \bigl\lVert ( I - \pi _h ) P _{\mathfrak{H}} \bigr\rVert $.
If the family of subspaces $ \{ V _h \} _{ h > 0 } $ is pointwise
approximating in $V$, then it follows that $ \{ W _h \} _{ h > 0 } $
is pointwise approximating in $W$, since
\begin{equation*}
  \inf _{ w _h \in W _h } \lVert w - w _h \rVert \leq \inf _{ v \in
    V } \bigl( \lVert w - v \rVert + \inf _{ v _h \in V _h } \lVert v
  - v _h \rVert _V \bigr) \rightarrow 0 ,
\end{equation*} 
by density of $V$ in $W$ together with the pointwise approximating
condition in $V$. Moreover, if the operators $ \pi _h $ are uniformly
bounded in $h$, then quasi-optimality implies that
$ I - \pi _h \rightarrow 0 $ pointwise in $W$.  Since $K$ and
$ P _{ \mathfrak{H} } $ are compact operators, they convert pointwise
convergence to norm convergence, and therefore
$ \eta, \mu \rightarrow 0 $. In the typical case of the de~Rham
complex, when $ V _h $ consists of piecewise polynomial differential
forms up to degree $r$, we will have $ \eta = O (h) $ and $\mu = O (h
^{ r + 1 } )$ (\citet[p.~312]{ArFaWi2010}).

In \autoref{thm:errorEstimate}, recall that we bounded the quantity $
\lVert u - u _h \rVert _V + \lVert p - p _h \rVert $. Refined
estimates will now be obtained by breaking this up into several
components,
\begin{equation*}
  \bigl\lVert  \mathrm{d} ( u - u _h ) \bigr\rVert , \qquad   \lVert P
  _{ \mathfrak{B}  } u - P _{ \mathfrak{B}  _h } u _h \rVert
  , \qquad \lVert P _{ \mathfrak{B}  ^\ast } u - P _{ \mathfrak{B}  _h
    ^\ast } u _h \rVert , \qquad \lVert p - p _h \rVert ,
\end{equation*} 
and estimating each of these individually, in a sequence of
theorems. Equivalently, these results can be interpreted as giving
error estimates for the individual terms of the discrete Hodge
decomposition. Before doing so, we begin with a lemma that will be
useful throughout this section.

\begin{lemma}
  \label{lem:improvedEstLemma}
  If $ v _h \in \mathfrak{Z}  _h ^\perp $ and $ v = P _{ \mathfrak{B}
    ^\ast } v _h $, then
  \begin{equation*}
    \lVert v - v _h \rVert \leq \bigl\lVert ( I - \pi _h ) v
    \bigr\rVert \leq \eta \lVert \mathrm{d} v _h \rVert .
  \end{equation*} 
\end{lemma}

\begin{proof}
  Since $ \pi _h v - v _h = \pi _h ( v - v _h ) \in \mathfrak{Z} _h
  \subset \mathfrak{Z} $, we have $ ( v _h - \pi _h v ) \perp ( v _h -
  v ) $. Therefore, the Pythagorean theorem implies
  \begin{equation*}
    \lVert v - v _h \rVert \leq \lVert v - \pi _h v \rVert =
    \bigl\lVert ( I - \pi _h ) v \bigr\rVert .
  \end{equation*} 
  Finally, observing that $ P _{ \mathfrak{B} ^\ast } = K \mathrm{d}
  $, it follows that $ v = K \mathrm{d} v _h $, and thus $ \bigl\lVert
  ( I - \pi _h ) v \bigr\rVert = \bigl\lVert ( I - \pi _h ) K
  \mathrm{d} v _h \bigr\rVert \leq \eta \lVert \mathrm{d} v _h \rVert
  $.
\end{proof}

For the following estimates, let $ P _h $ denote the $W$-orthogonal
projection onto $ V _h $, and for $ w \in W $, define
\begin{equation*}
  E (w) = \bigl\lVert ( I - P _h ) w \bigr\rVert = \inf _{ v _h \in V
    _h } \lVert w - v _h  \rVert ,
\end{equation*} 
i.e., the best approximation to $ w \in W $ by an element of $ V _h $.

\begin{theorem}
  \label{thm:duhEstimate}
  $ \bigl\lVert \mathrm{d} ( u - u _h ) \bigr\rVert \leq C E
  (\mathrm{d} u ) $.
\end{theorem}

\begin{proof}
  Since $ \mathrm{d} u _h = P _{ \mathfrak{B} _h } f = P _{
    \mathfrak{B} _h } P _{ \mathfrak{B} } f = P _{ \mathfrak{B} _h }
  \mathrm{d} u $ and $ \pi _h \mathrm{d} u \in \mathfrak{B} _h $, it
  follows that $ \bigl\lVert \mathrm{d} ( u - u _h ) \bigr\rVert =
  \bigl\lVert ( I - P _{ \mathfrak{B} _h } ) \mathrm{d} u \bigr\rVert
  \leq \bigl\lVert ( I - \pi _h ) \mathrm{d} u \bigr\rVert \leq C E (
  \mathrm{d} u ) $.
\end{proof}

\begin{theorem}
  $ \lVert P _{ \mathfrak{B}  } u - P _{ \mathfrak{B}  _h } u _h
  \rVert \leq C \Bigl( E ( P _{ \mathfrak{B}  } u ) + \eta \bigl[ E (
  \mathrm{d} u ) + E ( p ) \bigr] \Bigr) $.
\end{theorem}

\begin{proof}
  We begin by writing
  \begin{align*}
    \lVert P _{ \mathfrak{B}  } u - P _{ \mathfrak{B}  _h } u _h
    \rVert &\leq \lVert P _{ \mathfrak{B}  } u - P _{ \mathfrak{B}
      _h } P _{ \mathfrak{B}  } u \rVert + \lVert P _{
      \mathfrak{B}  _h } P _{ \mathfrak{B}  } u - P _{ \mathfrak{B}
      _h } u _h \rVert \\
    &\leq \bigl\lVert ( I - P _{ \mathfrak{B}  _h } ) P _{
      \mathfrak{B}  } u \bigr\rVert + \bigl\lVert P _{ \mathfrak{B}
      _h } ( u - u _h ) \bigr\rVert .
  \end{align*} 
  For the first term, optimality of $ P _{ \mathfrak{B} _h } $ and
  quasi-optimality of $ \pi _h $, together with the fact that $ \pi _h
  P _{ \mathfrak{B} } u \in \mathfrak{B} _h $, implies
  \begin{equation*}
    \bigl\lVert ( I - P _{ \mathfrak{B}  _h } ) P _{ \mathfrak{B}  } u
    \bigr\rVert \leq \bigl\lVert ( I - \pi _h ) P _{ \mathfrak{B}  } u
    \bigr\rVert \leq C E ( P _{ \mathfrak{B}  } u ) .
  \end{equation*} 
  The remaining term will now be bounded using a duality-type
  argument.

  Let $ e = P _{ \mathfrak{B} _h } ( u - u _h ) $, $ w = K e $, and $
  w _h = K _h e $. Since $ e \in \mathfrak{B} _h \subset \mathfrak{B}
  $, we have $ e = \mathrm{d} w = \mathrm{d} \pi _h w = \mathrm{d} w
  _h $. In particular, this implies that $ \mathrm{d} ( \pi _h w - w
  _h ) = 0 $, so $ \pi _h w - w _h \in \mathfrak{Z} _h \subset
  \mathfrak{Z} $ is orthogonal to both $ w \in \mathfrak{B} ^\ast $
  and $ w _h \in \mathfrak{B} _h ^\ast $, and hence to $ w - w _h
  $. Therefore, the Pythagorean theorem and Lemma
  \ref{lem:improvedEstLemma} imply that
  \begin{equation*}
    \lVert w - w _h \rVert \leq \bigl\lVert ( I - \pi _h ) w
    \bigr\rVert = \bigl\lVert ( I - \pi _h ) K e \bigr\rVert \leq
    \eta \lVert e \rVert .
  \end{equation*} 
  Now, using $ e = \mathrm{d} w _h $ and the variational principles
  \eqref{eqn:hodgeDirac} and \eqref{eqn:discreteHodgeDirac}, we have
  \begin{equation*} 
    \lVert e \rVert ^2 = \langle e, \mathrm{d} w _h \rangle = \langle
    u - u _h, \mathrm{d} w _h \rangle 
    = - \bigl\langle \mathrm{d} ( u
    - u _h ) + ( p - p _h ) , w _h \bigr\rangle .
  \end{equation*} 
  Furthermore, since $ w _h \in \mathfrak{B} _h ^\ast \perp
  \mathfrak{H} _h $ and $ w \in \mathfrak{B} ^\ast \perp \mathfrak{Z}
  $, we can write this as
  \begin{align*}
    \lVert e \rVert ^2 &= - \bigl\langle \mathrm{d} ( u - u _h ) + ( p
    - P _{ \mathfrak{H}  _h } p ) , w _h - w \bigr\rangle \\
    &\leq \Bigl[ \bigl\lVert \mathrm{d} ( u - u _h ) \bigr\rVert +
    \bigl\lVert ( I - P _{ \mathfrak{H} _h } ) p \bigr\rVert \Bigr]
    \lVert w _h - w \rVert.
  \end{align*} 
  We already have the estimate $ \bigl\lVert \mathrm{d} ( u - u _h )
  \bigr\rVert \leq C E ( \mathrm{d} u ) $, while
  \begin{equation*}
    \bigl\lVert ( I - P _{ \mathfrak{H}  _h } ) p \bigr\rVert =
    \bigl\lVert ( I - P _{ \mathfrak{Z}  _h } ) p \bigr\rVert \leq
    \bigl\lVert ( I - \pi _h ) p \bigr\rVert \leq C E (p) .
  \end{equation*} 
  Combining these with $ \lVert w - w _h \rVert \leq \eta \lVert e
  \rVert $ and dividing through by $ \lVert e \rVert $, we finally
  obtain
  \begin{equation*}
    \lVert e \rVert \leq C \eta  \bigl[ E ( \mathrm{d} u ) + E (p)
    \bigr] ,
  \end{equation*} 
  which completes the proof.
\end{proof}
 
\begin{theorem}
  $ \lVert P _{ \mathfrak{B}  ^\ast } u - P _{ \mathfrak{B}  _h ^\ast
  } u _h \rVert \leq C \bigl[ E ( P _{ \mathfrak{B}  ^\ast } u ) +
  \eta E ( \mathrm{d} u ) \bigr] $.
\end{theorem}

\begin{proof}
  As in the previous proof, we begin by using the triangle inequality
  to split this into two pieces,
  \begin{align*}
    \lVert P _{ \mathfrak{B} ^\ast } u - P _{ \mathfrak{B} _h ^\ast }
    u _h \rVert &\leq \lVert P _{ \mathfrak{B} ^\ast } u - P _{
      \mathfrak{B} _h ^\ast } P _{ \mathfrak{B} ^\ast } u \rVert +
    \lVert P _{ \mathfrak{B} _h ^\ast } P _{ \mathfrak{B} ^\ast } u -
    P _{ \mathfrak{B}  _h ^\ast } u _h \rVert\\
    &= \bigl\lVert ( I - P _{ \mathfrak{B} _h ^\ast } ) P _{
      \mathfrak{B} ^\ast } u \bigr\rVert + \bigl\lVert P _{
      \mathfrak{B} _h ^\ast } ( P _{ \mathfrak{B} ^\ast } u - u _h )
    \bigr\rVert .
  \end{align*} 
  Observe that, since $ \mathfrak{Z} _h \subset \mathfrak{Z} \perp
  \mathfrak{B} ^\ast $, it follows that
  \begin{equation*}
    P _h P _{ \mathfrak{B}  ^\ast } u = P _{ \mathfrak{Z}  _h } P _{
      \mathfrak{B}  ^\ast } u + P _{ \mathfrak{B}  ^\ast _h } P _{
      \mathfrak{B}  ^\ast } u = P _{ \mathfrak{B}  _h ^\ast } P _{
      \mathfrak{B}  ^\ast } u .
  \end{equation*} 
  Therefore,
  \begin{equation*}
    \bigl\lVert ( I - P _{ \mathfrak{B}  _h ^\ast } ) P _{
      \mathfrak{B}  ^\ast } u \bigr\rVert =     \bigl\lVert ( I - P _h
    ) P _{  \mathfrak{B}  ^\ast } u \bigr\rVert = E ( P _{
      \mathfrak{B}  ^\ast } u ) .
  \end{equation*} 
  For the second piece, let $ v _h = P _{ \mathfrak{B} _h ^\ast } (
  \pi _h P _{ \mathfrak{B} ^\ast } u - u _h ) $, so that
  \begin{equation*}
    \bigl\lVert P _{  \mathfrak{B}  _h ^\ast } ( P _{ \mathfrak{B}
      ^\ast } u - u _h ) \bigr\rVert \leq \bigl\lVert P _{
      \mathfrak{B}  _h ^\ast } ( I - \pi _h ) P _{ \mathfrak{B}  ^\ast
    } u \bigr\rVert  + \lVert v _h \rVert \leq C E ( P _{ \mathfrak{B}
      ^\ast } u ) + \lVert v _h \rVert .
  \end{equation*} 
  It now suffices to control $ \lVert v _h \rVert $.

  To do so, we let $ v = P _{ \mathfrak{B} ^\ast } v _h $, as in
  \autoref{lem:improvedEstLemma}, and observe that
  \begin{equation*}
    \langle v , v _h \rangle = \bigl\langle v, v _h + P _{
      \mathfrak{Z}  _h } ( \pi _h P _{ \mathfrak{B} ^\ast } u - u
    _h ) \bigr\rangle = \langle v, \pi _h P _{ \mathfrak{B}  ^\ast } u - u
    _h \rangle .
  \end{equation*} 
  Therefore,
  \begin{align*}
    \lVert v _h \rVert ^2 &= \langle v _h - v , v _h \rangle + \langle
    v, v _h \rangle \\
    &= \langle v _h - v , v _h \rangle + \langle v, \pi _h P _{
      \mathfrak{B} ^\ast } u - u _h \rangle \\
    &= \langle v _h - v , v _h \rangle + \langle v, \pi _h P _{
      \mathfrak{B} ^\ast } u - P _{ \mathfrak{B} ^\ast } u \rangle +
    \langle v, P _{ \mathfrak{B} ^\ast } u - u _h \rangle .
  \end{align*} 
  For the first two terms, \autoref{lem:improvedEstLemma} implies
  \begin{align*}
    \langle v _h - v , v _h \rangle + \langle v, \pi _h P _{
      \mathfrak{B} ^\ast } u - P _{ \mathfrak{B} ^\ast } u \rangle
    &\leq \lVert v _h - v \rVert \lVert v _h \rVert + \lVert v \rVert
    \bigl\lVert ( I - \pi _h ) P _{ \mathfrak{B} ^\ast } u
    \bigr\rVert \\
    &\leq \bigl[ \eta \lVert \mathrm{d} v _h \rVert + C E ( P _{
      \mathfrak{B} ^\ast } u ) \bigr] \lVert v _h \rVert .
  \end{align*} 
  Now, $ \mathrm{d} v _h = \pi _h \mathrm{d} u - \mathrm{d} u _h $, so
  \autoref{thm:duhEstimate} implies that
  \begin{equation*}
    \lVert \mathrm{d} v _h \rVert \leq \bigl\lVert ( I - \pi _h )
    \mathrm{d} u \bigr\rVert + \bigl\lVert \mathrm{d} ( u - u _h )
    \bigr\rVert \leq C E ( \mathrm{d} u ) .
  \end{equation*} 
  Finally, it remains to control $ \langle v, u - u _h \rangle
  $. Since $ v \in \mathfrak{B} ^\ast $ and $ P _{ \mathfrak{B} ^\ast
  } = K \mathrm{d} $,
  \begin{align*}
    \langle v, u - u _h \rangle &= \langle K v , \mathrm{d} ( u - u _h
    ) \bigr\rangle \\
    &= \bigl\langle ( I - \pi _h ) K v , \mathrm{d} ( u - u _h )
    \bigr\rangle + \bigl\langle \pi _h K v , \mathrm{d} ( u - u _h )
    \bigr\rangle .
  \end{align*} 
  For the first term, another application of \autoref{thm:duhEstimate}
  yields
 \begin{equation*}
    \bigl\langle ( I - \pi _h ) K v , \mathrm{d} ( u - u _h )
    \bigr\rangle \leq \bigl\lVert ( I - \pi _h ) K v \bigr\rVert
    \bigl\lVert \mathrm{d} ( u - u _h ) \bigr\rVert \leq C \eta E (
    \mathrm{d} u ) \lVert v _h \rVert .
  \end{equation*} 
  On the other hand, the second term vanishes: since $ v \in
  \mathfrak{B} ^\ast $, we have $ K v \in \mathfrak{B} $ and $ \pi _h
  K v \in \mathfrak{B} _h $, while $ P _{ \mathfrak{B} _h } (
  \mathrm{d} u - \mathrm{d} u _h ) = \mathrm{d} u _h - \mathrm{d} u _h
  = 0 $. Hence, we have finally shown that $ \lVert v _h \rVert \leq C
  \bigl[ E ( P _{ \mathfrak{B} ^\ast } u ) + \eta E ( \mathrm{d} u )
  \bigr] $, so the result follows.
\end{proof}

\begin{theorem}
  $ \lVert p - p _h \rVert \leq C \bigl[ E (p) + \mu E ( \mathrm{d} u
  ) \bigr] $.
\end{theorem}

\begin{proof}
  We begin by decomposing
  \begin{equation*}
    \lVert p - p _h \rVert ^2 = \lVert p - P _{ \mathfrak{H}  _h } p
    \rVert ^2 + \lVert P _{ \mathfrak{H}  _h } p - p _h \rVert ^2 .
  \end{equation*} 
  First, the inequality \eqref{eqn:p-q} implies that $ \lVert p - P _{
    \mathfrak{H} _h } p \rVert \leq \bigl\lVert ( I - \pi _h ) p
  \bigr\rVert \leq C E (p) $.  For the second term, we observe that $
  p = P _\mathfrak{H} f $, while $ p _h = P _{ \mathfrak{H} _h } f
  $. Therefore,
  \begin{equation*}
    p _h - P _{ \mathfrak{H}  _h } p = P _{ \mathfrak{H}  _h } ( I - P
    _{ \mathfrak{H}  } ) f = P _{ \mathfrak{H}  _h } ( P _{
      \mathfrak{B}  } + P _{ \mathfrak{B}  ^\ast } ) f = P _{
      \mathfrak{H}  _h } P _{ \mathfrak{B}  } f = P _{ \mathfrak{H}
      _h } \mathrm{d} u ,
  \end{equation*} 
  where we have used the fact that $ P _{ \mathfrak{H} _h } P _{
    \mathfrak{B} ^\ast } = 0 $, since $ \mathfrak{H} _h \subset
  \mathfrak{Z} _h \subset \mathfrak{Z} \perp \mathfrak{B} ^\ast
  $. Now, using the orthogonality of the Hodge decomposition and the
  fact that $ \pi _h \mathrm{d} u = \mathrm{d} \pi _h u \in
  \mathfrak{B} _h \subset \mathfrak{B} $, we get
  \begin{align*}
    \lVert P _{ \mathfrak{H} _h } \mathrm{d} u \rVert ^2 &= \langle
    \mathrm{d} u , P _{ \mathfrak{H} _h } \mathrm{d} u
    \rangle \\
    &= \bigl\langle ( I - \pi _h ) \mathrm{d} u , ( I - P _{
      \mathfrak{H} } ) P _{ \mathfrak{H} _h } \mathrm{d} u
    \bigr\rangle \\
    &\leq \bigl\lVert ( I - \pi _h ) \mathrm{d} u \bigr\rVert
    \bigl\lVert ( I - P _{ \mathfrak{H} } ) P _{ \mathfrak{H} _h }
    \mathrm{d} u \bigr\rVert .
  \end{align*} 
  Quasi-optimality of $ \pi _h $ immediately implies $ \bigl\lVert ( I
  - \pi _h ) \mathrm{d} u \bigr\rVert \leq C E ( \mathrm{d} u ) $. On
  the other hand, since $ P _{ \mathfrak{H} _h } \mathrm{d} u \in
  \mathfrak{Z} $, the Hodge decomposition implies that $ ( I - P _{
    \mathfrak{H} } ) P _{ \mathfrak{H} _h } \mathrm{d} u \in
  \mathfrak{B} $. Hence, $ \pi _h ( I - P _{ \mathfrak{H} } ) P _{
    \mathfrak{H} _h } \mathrm{d} u \in \mathfrak{B} _h $ is orthogonal
  to both $ P _{ \mathfrak{H} _h } \mathrm{d} u \in \mathfrak{H} _h $
  and $ P _{ \mathfrak{H} } P _{ \mathfrak{H} _h } \mathrm{d} u $, and
  thus to $ ( I - P _{ \mathfrak{H} } ) P _{ \mathfrak{H} _h }
  \mathrm{d} u $. Therefore, the Pythagorean theorem gives
  \begin{align*}
    \bigl\lVert ( I - P _{ \mathfrak{H} } ) P _{ \mathfrak{H} _h }
    \mathrm{d} u \bigr\rVert &\leq \bigl\lVert ( I - \pi _h ) ( I - P
    _{ \mathfrak{H} } ) P _{ \mathfrak{H} _h } \mathrm{d} u
    \bigr\rVert \\
    &= \bigl\lVert ( I - \pi _h ) P _{ \mathfrak{H} } P _{
      \mathfrak{H} _h } \mathrm{d} u \bigr\rVert \\
    &\leq \mu \bigl\lVert P _{ \mathfrak{H}  _h } \mathrm{d} u
    \bigr\rVert .
  \end{align*} 
  Finally, this shows that $ \lVert P _{ \mathfrak{H} _h } p - p _h
  \rVert \leq C \mu E ( \mathrm{d} u ) $, so the proof is complete.
\end{proof}
  
\section{Relationship to the Hodge--Laplace problem}
\label{sec:laplace}

In this section, we discuss the close relationship between the
Hodge--Dirac problem and the Hodge--Laplace problem, which has been
the main focus of much of the extant work on finite element exterior
calculus (particularly \citet{ArFaWi2010}). In fact, we show that some
of the key estimates of finite element exterior calculus may be
recovered as direct corollaries of the stability and convergence
results presented in the previous two sections.

\subsection{The Hodge--Laplace operator and mixed variational problem}

The \emph{abstract Hodge--Laplace operator} is $ L = \mathrm{D} ^2 =
\mathrm{d} \mathrm{d} ^\ast + \mathrm{d} ^\ast \mathrm{d} $, which is
defined on the domain $ \mathcal{D} (L) = \mathrm{D} ^{-1} ( V \cap V
^\ast ) \subset V \cap V ^\ast $ and has kernel $ \mathcal{N} (L) =
\mathcal{N} (\mathrm{D}) = \mathfrak{H} $. The Hodge--Laplace problem
is then the following: Given $ f \in W $, find $ ( u , p ) \in \bigl(
\mathcal{D} (L) \cap \mathcal{N} (L) ^\perp \bigr) \oplus \mathcal{N}
(L) $ such that
\begin{equation*}
  L u + p = f .
\end{equation*} 
To solve this, we may simply solve the Hodge--Dirac problem $
\mathrm{D} w + p = f $, and then solve another Hodge--Dirac problem $
\mathrm{D} u = w $. (Since $ w \perp \mathfrak{H} $, the harmonic part
is omitted from the second problem.) Therefore,
\begin{equation*}
  L u + p = \mathrm{D} ^2 u + p = \mathrm{D} ( \mathrm{D} u ) + p =
  \mathrm{D} w + p = f ,
\end{equation*} 
so $ ( u,p) $ is indeed a solution to the Hodge--Laplace problem, as
claimed.

To apply finite-element techniques to this problem, it must be put
into a variational form. Naively, we might think to use the
variational problem: Find $ ( u, p ) \in ( V \cap V ^\ast ) \times
\mathfrak{H} $ such that
\begin{equation*}
  \begin{alignedat}{2}
    \langle \mathrm{d} ^\ast u, \mathrm{d} ^\ast v \rangle + \langle
    \mathrm{d} u, \mathrm{d} v \rangle + \langle p, v \rangle &=
    \langle f, v \rangle , \quad &\forall v &\in V \cap V ^\ast  ,\\
    \langle u, q \rangle &= 0 , \quad &\forall q &\in \mathfrak{H} .
  \end{alignedat}
\end{equation*}
This problem is easily shown to be well-posed, by virtue of the
boundedness, symmetry, and coercivity of the bilinear form $ ( u, p;
v, q) \mapsto \langle \mathrm{d} ^\ast u , \mathrm{d} ^\ast v \rangle
+ \langle \mathrm{d} u , \mathrm{d} v \rangle + \langle p , v \rangle
+ \langle u , q \rangle $. In fact, the bilinear form gives an
equivalent inner product for the Hilbert space $ ( V \cap V ^\ast )
\times \mathfrak{H} $, so this is a simple application of the Riesz
representation theorem to the functional $ ( v, q ) \mapsto \langle f,
v \rangle $. Yet, despite the \emph{stability} of this formulation, it
is generally unsuitable for the finite element method due to
\emph{consistency} difficulties. In particular, for the case of the
Hodge--de~Rham complex, piecewise polynomial differential forms are
too regular to approximate singular elements of $ V \cap V ^\ast $,
cf.~\citet{BrFo1991,ArFaWi2010}.

Instead, it is preferable to use the following mixed variational
formulation: Find $ ( \sigma, u , p ) \in V \times V \times
\mathfrak{H} $ such that
\begin{equation}
  \label{eqn:hodgeLaplace}
  \begin{alignedat}{2}
    \langle \sigma ,\tau \rangle - \langle u , \mathrm{d} \tau \rangle
    &= 0 , \quad &\forall \tau &\in V,\\
    \langle \mathrm{d} \sigma, v \rangle + \langle \mathrm{d} u,
    \mathrm{d} v \rangle + \langle p, v \rangle &=
    \langle f, v \rangle , \quad &\forall v &\in V ,\\
    \langle u, q \rangle &= 0 , \quad &\forall q &\in \mathfrak{H} .
  \end{alignedat}
\end{equation}
Indeed, if $ ( \sigma, u , p ) $ is a solution, then the first line of
\eqref{eqn:hodgeLaplace} implies $ \sigma = \mathrm{d} ^\ast u $,
while the second line implies that
\begin{equation*}
  L u + p = \mathrm{d} \mathrm{d} ^\ast u + \mathrm{d} ^\ast
  \mathrm{d} u + p = \mathrm{d} \sigma + \mathrm{d} ^\ast \mathrm{d} u
  + p = f .
\end{equation*} 
(Note that $ \mathrm{d} \sigma + \mathrm{d} ^\ast \mathrm{d} u + p = f
$ is precisely the Hodge decomposition of $f$.) We now show that---as
with the non-variational form of the problems---we may solve the
Hodge--Laplace mixed variational problem \eqref{eqn:hodgeLaplace} by
simply solving the Hodge--Dirac problem \eqref{eqn:hodgeDirac} twice.

\begin{theorem}
  \label{thm:diracLaplace}
  If $ ( w, p ) \in V \times \mathfrak{H} $ solves the Hodge--Dirac
  variational problem for $f$, and $ ( u, 0 ) \in V \times
  \mathfrak{H} $ solves the Hodge--Dirac variational problem for $w$,
  then $ ( w - \mathrm{d} u , u , p ) \in V \times V \times
  \mathfrak{H} $ solves \eqref{eqn:hodgeLaplace}.
\end{theorem}

\begin{proof}
  For any $ \tau \in V $, we have
  \begin{equation*}
    \langle w - \mathrm{d} u , \tau \rangle - \langle u , \mathrm{d}
    \tau \rangle = \langle w, \tau \rangle - \bigl(  \langle
    \mathrm{d} u , \tau \rangle + \langle u , \mathrm{d} \tau \rangle
    \bigr)  = \langle w, \tau \rangle - \langle w, \tau \rangle = 0 .
  \end{equation*} 
  Next, for any $ v \in V $, we have
  \begin{align*}
    \bigl\langle \mathrm{d} ( w - \mathrm{d} u ) , v \bigr\rangle +
    \langle \mathrm{d} u, \mathrm{d} v \rangle + \langle p, v \rangle
    &= \langle \mathrm{d} w , v \rangle + \langle \mathrm{d} u ,
    \mathrm{d} v \rangle + \langle p , v \rangle \\
    &= \langle \mathrm{d} w , v \rangle + \bigl( \langle \mathrm{d} u
    , \mathrm{d} v \rangle + \langle u , \mathrm{d} \mathrm{d} v
    \rangle \bigr) + \langle p , v \rangle \\
    &= \langle \mathrm{d} w , v \rangle + \langle w, \mathrm{d} v
    \rangle + \langle p, v \rangle \\
    &= \langle f, v \rangle .
  \end{align*} 
  Finally, since $( u, 0 ) \in V \times \mathfrak{H} $ solves the
  Hodge--Dirac variational problem for $w$, we have $ \langle u, q
  \rangle = 0 $ for all $ q \in \mathfrak{H} $, which completes the
  proof.
\end{proof}

As an immediate consequence of this fact, together with the
well-posedness result \autoref{cor:hodgeDiracWellPosed} for
Hodge--Dirac, we obtain well-posedness for Hodge--Laplace. (Compare
\citet[Theorem 3.1]{ArFaWi2010}.)

\begin{theorem}
  \label{thm:hodgeLaplaceWellPosed}
  The mixed variational problem for the abstract Hodge--Laplace
  operator is well-posed.  That is, there exists a constant $ c
  ^\prime $, depending only on the Poincar\'e constant $ c _P $, such
  that for all $ f \in W $, the problem \eqref{eqn:hodgeLaplace} has a
  unique solution $ ( \sigma, u, p ) \in V \times \mathfrak{H} $,
  which satisfies the estimate
  \begin{equation*}
    \lVert \sigma \rVert _V + \lVert u \rVert _V + \lVert p \rVert
    \leq c ^\prime \lVert f \rVert .
  \end{equation*}
\end{theorem}

\begin{proof}
  By \autoref{thm:diracLaplace} and the triangle inequality,
  \begin{equation*}
    \lVert \sigma \rVert _V + \lVert u \rVert _V + \lVert p \rVert =
    \lVert w - \mathrm{d} u \rVert _V + \lVert u \rVert _V + \lVert p
    \rVert \leq \lVert w \rVert _V + 2 \lVert u  \rVert _V + \lVert p
    \rVert .
  \end{equation*} 
  Now, since $ ( u, 0 ) $ solves the Hodge--Dirac problem for $w$,
  \autoref{cor:hodgeDiracWellPosed} implies that $ \lVert u \rVert _V
  \leq c \lVert w \rVert \leq c \lVert w \rVert _V $, where $c$
  depends only on $ c _P $. Therefore,
  \begin{equation*}
    \lVert \sigma \rVert _V + \lVert u \rVert _V + \lVert p \rVert
    \leq \lVert w \rVert _V + 2 c \lVert w \rVert _V + \lVert p \rVert
    \leq ( 1 + 2 c ) \bigl( \lVert w \rVert _V + \lVert p \rVert
    \bigr) .
  \end{equation*} 
  Another application of \autoref{cor:hodgeDiracWellPosed} gives $
  \lVert w \rVert _V + \lVert p \rVert \leq c \lVert f \rVert $, so
  \begin{equation*}
    \lVert \sigma \rVert _V + \lVert u \rVert _V + \lVert p \rVert
    \leq ( 1 + 2 c ) c \lVert f \rVert = c ^\prime \lVert f \rVert ,
  \end{equation*} 
  which completes the proof.
\end{proof}

While we have shown that one can solve the Hodge--Laplace problem by
solving related Hodge--Dirac problems, the converse is also true: if $
L u + p = f $, then clearly $ \mathrm{D} ( \mathrm{D} u ) + p = f $,
so $ ( \mathrm{D} u , p ) $ solves the Hodge--Dirac problem. (Note
that $ \mathrm{D} u \perp \mathfrak{H} $, since the range of $
\mathrm{D} $ is $ \mathfrak{B} \oplus \mathfrak{B} ^\ast =
\mathfrak{H} ^\perp $.) A similar result holds for the variational
problems, as we now show.

\begin{theorem}
  If $ ( \sigma, u , p ) \in V \times V \times \mathfrak{H} $ solves
  \eqref{eqn:hodgeLaplace}, then $ ( \sigma + \mathrm{d} u , p ) \in V
  \times \mathfrak{H} $ solves \eqref{eqn:hodgeDirac}.
\end{theorem}

\begin{proof}
  For any $ v \in V $, since $ \mathrm{d} \mathrm{d} u = 0 $, we have
  $ \bigl\langle \mathrm{d} ( \sigma + \mathrm{d} u ) , v \rangle =
  \langle \mathrm{d} \sigma, v \rangle $. Similarly, by the first line
  of \eqref{eqn:hodgeLaplace} with $ \tau = \mathrm{d} v $, we have $
  \langle \sigma , \mathrm{d} v \rangle = \langle u , \mathrm{d}
  \mathrm{d} v \rangle = 0 $, so $ \langle \sigma + \mathrm{d} u ,
  \mathrm{d} v \rangle = \langle \mathrm{d} u , \mathrm{d} v \rangle
  $. Therefore, by the second line of \eqref{eqn:hodgeLaplace},
  \begin{equation*}
    \bigl\langle \mathrm{d} ( \sigma + \mathrm{d} u ) , v \bigr\rangle
    + \langle \sigma + \mathrm{d} u , \mathrm{d} v \rangle + \langle p
    , v \rangle = \langle \mathrm{d} \sigma , v \rangle + \langle
    \mathrm{d} u , \mathrm{d} v \rangle + \langle p , v \rangle =
    \langle f , v \rangle .
  \end{equation*} 
  Finally, for any $ q \in \mathfrak{H} $, the first line of
  \eqref{eqn:hodgeLaplace} gives $ \langle \sigma, q \rangle = \langle
  u , \mathrm{d} q \rangle = 0 $, since $ \mathfrak{H} \subset
  \mathfrak{Z} $. Furthermore, we also have $ \langle \mathrm{d} u , q
  \rangle = 0 $, since $ \mathfrak{H} \perp \mathfrak{B} $. Hence,
  \begin{equation*}
    \langle \sigma + \mathrm{d} u , q \rangle = 0 ,
  \end{equation*} 
  which completes the proof.
\end{proof}

\subsection{The discrete Hodge--Laplace problem}

Again, let $ V _h \subset V $ be a closed subspace satisfying the
conditions discussed in \autoref{sec:discrete} (i.e., $ \mathrm{d} _h
= \mathrm{d} \rvert _{ V _h } $ is a closed-nilpotent operator on $ V
_h $) and equipped with a bounded commuting projection $ \pi _h $. We
then consider the discrete Hodge--Laplace mixed variational problem:
Find $ ( \sigma _h , u _h , p _h ) \in V _h \times V _h \times
\mathfrak{H} _h $ such that
\begin{equation}
  \label{eqn:discreteHodgeLaplace}
  \begin{alignedat}{2}
    \langle \sigma _h ,\tau _h \rangle - \langle u _h , \mathrm{d}
    \tau _h \rangle
    &= 0 , \quad &\forall \tau _h &\in V _h ,\\
    \langle \mathrm{d} \sigma _h , v _h \rangle + \langle \mathrm{d} u
    _h , \mathrm{d} v _h \rangle + \langle p _h , v _h \rangle &=
    \langle f, v _h \rangle , \quad &\forall v _h &\in V _h ,\\
    \langle u _h , q _h \rangle &= 0 , \quad &\forall q _h &\in
    \mathfrak{H} _h .
  \end{alignedat}
\end{equation}
Applying \autoref{thm:diracLaplace} to $ V _h $ implies that this
discrete Hodge--Laplace problem can be solved by solving two discrete
Hodge--Dirac problems \eqref{eqn:discreteHodgeDirac} sequentially.
That is, if $ ( w _h , p _h ) \in V _h \times \mathfrak{H} _h $ solves
the discrete Hodge--Dirac problem for $f$, and $ ( u _h , 0 ) \in V _h
\times \mathfrak{H} _h $ solves the discrete Hodge--Dirac problem for
$w _h $, then $ ( w _h - \mathrm{d} u _h , u _h , p _h ) \in V _h
\times V _h \times \mathfrak{H} _h $ solves
\eqref{eqn:discreteHodgeLaplace}. Thus, as in
\autoref{thm:hodgeLaplaceWellPosed}, the discrete Hodge--Laplace
problem is well-posed, and moreover
\autoref{cor:discreteHodgeDiracWellPosed} implies that the constant $
c ^\prime _h $ depends only on $ c _P $ and $ \lVert \pi _h \rVert $.

As a result, it is now straightfoward to obtain an error estimate for
the approximation of the Hodge--Laplace problem
\eqref{eqn:hodgeLaplace} by the discrete problem
\eqref{eqn:discreteHodgeLaplace}, which follows from the analogous
Hodge--Dirac result, \autoref{thm:errorEstimate}.  (Compare
\citet[Theorem 3.9]{ArFaWi2010}.)

\begin{theorem}
  Under the hypotheses of \autoref{thm:errorEstimate}, if $ ( \sigma ,
  u , p ) \in V \times V \times \mathfrak{H} $ solves
  \eqref{eqn:hodgeLaplace} and $ ( \sigma _h , u _h , p _h ) \in V _h
  \times V _h \times \mathfrak{H} _h $ solves
  \eqref{eqn:discreteHodgeLaplace}, then we have the error estimate
  \begin{multline*}
    \lVert \sigma - \sigma _h \rVert _V + \lVert u - u _h \rVert _V +
    \lVert p - p _h \rVert \\
    \leq C \biggl( \inf _{ \tau \in V _h } \lVert \sigma - \tau \rVert _V
    + \inf _{ v \in V _h } \lVert u - v \rVert _V + \inf _{ q \in V _h
    } \lVert p - q \rVert _V + \mu \inf _{ v \in V _h } \lVert P _{
      \mathfrak{B} } u - v \rVert _V \biggr) .
  \end{multline*} 
\end{theorem}

\begin{proof}
  By \autoref{thm:diracLaplace}, we have $ \sigma = w - \mathrm{d} u $
  and $ \sigma _h = w _h - \mathrm{d} u _h $, where $(w,p)$ and $ (w
  _h, p _h ) $ solve, respectively, the Hodge--Dirac and discrete
  Hodge--Dirac problems for $f$. Therefore, by the triangle inequality,
  \begin{equation*}
    \lVert \sigma - \sigma _h \rVert _V + \lVert u - u _h \rVert _V +
    \lVert p - p _h \rVert \leq C \bigl( \lVert w - w _h \rVert _V +
    \lVert u - u _h \rVert _V + \lVert p - p _h \rVert \bigr) .
  \end{equation*} 
  Now, by \autoref{thm:errorEstimate},
  \begin{equation*}
    \lVert w - w _h \rVert _V + \lVert p - p _h \rVert \\
    \leq C \biggl( \inf _{ v \in V _h } \lVert w - v \rVert _V + \inf _{
      q \in V _h } \lVert p - q \rVert _V + \mu \inf _{ v \in V _h }
    \lVert P _{ \mathfrak{B} } w - v \rVert _V \biggr) .
  \end{equation*} 
  Writing $ w = \sigma + \mathrm{d} u $, we then get
  \begin{equation*}
    \inf _{ v \in V _h } \lVert w - v \rVert _V \leq \inf _{ \tau \in V _h
    } \lVert \sigma - \tau \rVert _V + \inf _{ v \in V _h } \lVert u - v
    \rVert _V .
  \end{equation*} 
  Moreover, the first line of the Hodge--Laplace variational principle
  \eqref{eqn:hodgeLaplace} implies that $ P _{ \mathfrak{B} } \sigma =
  0 $, so $ P _{ \mathfrak{B} } w = \mathrm{d} u $, and thus
  \begin{equation*}
    \inf _{ v \in V _h } \lVert P _{ \mathfrak{B}  } w - v \rVert _V
    \leq \inf _{ v \in V _h } \lVert u - v \rVert _V ,
  \end{equation*} 
  Altogether, we now have 
  \begin{equation*}
    \lVert w - w _h \rVert _V + \lVert p - p _h \rVert \\
    \leq C \biggl( \inf _{ \tau \in V _h } \lVert \sigma - \tau \rVert _V +
    \inf _{ v \in V _h } \lVert u - v \rVert _V + \inf _{ q \in V _h }
    \lVert p - q \rVert _V \biggr) .
  \end{equation*}
  so it suffices to control the remaining term $ \lVert u - u _h
  \rVert _V $.
  
  Recall that $ ( u, 0 ) $ solves the Hodge--Dirac problem for $w$,
  while $ ( u _h , 0 ) $ solves the Hodge--Dirac problem for $ w _h
  $. Since the right-hand sides of these problems are different, it is
  not possible to apply \autoref{thm:errorEstimate} just yet: there is
  a ``variational crime'' that must be controlled. To do so, let $ ( u
  _h ^\prime , P _{ \mathfrak{H} _h } w ) $ be the solution of the
  discrete Hodge--Dirac problem for $w$, so using the triangle
  inequality,
  \begin{equation*}
    \lVert u - u _h \rVert _V \leq \lVert u - u _h ^\prime \rVert _V +
    \lVert u _h ^\prime - u _h \rVert _V .
  \end{equation*} 
  By \autoref{cor:discreteHodgeDiracWellPosed}, the discrete
  well-posedness result, we have $ \lVert u _h ^\prime - u _h \rVert
  _V \leq C \lVert w - w _h \rVert $, which is already under control,
  while another application of \autoref{thm:errorEstimate} gives
  \begin{equation*}
    \lVert u - u _h ^\prime \rVert _V \leq \lVert u - u _h ^\prime
    \rVert _V + \lVert P _{ \mathfrak{H}  _h } w \rVert \leq C
    \biggl( \inf _{ v \in V _h } \lVert u - v \rVert _V + \inf _{ v
      \in V _h } \lVert P _{  \mathfrak{B}  } u - v \rVert _V \biggr).
  \end{equation*} 
  Finally, combining these estimates yields the claimed result.
\end{proof}
  
\section{A note on Hilbert complexes and nilpotent operators}
\label{sec:complex}

The framework of finite element exterior calculus, as developed by
\citet{ArFaWi2010}, is described using \emph{Hilbert complexes}
(cf.~\citet{BrLe1992}), which are closely related to the nilpotent
operators considered here (and also by
\citet{AxMc2004,AxKeMc2006}). Rather than a single Hilbert space $W$
and nilpotent operator $ \mathrm{d} $, a Hilbert complex consists of a
sequence of several Hilbert spaces $ W ^k $ and closed,
densely-defined operators
$ \mathrm{d} ^k \colon V ^k \subset W ^k \rightarrow V ^{ k + 1 }
\subset W ^{ k + 1 } $
satisfying the nilpotency property
$ \mathrm{d} ^k \mathrm{d} ^{k-1} = 0 $. The domains and operators
thus form a cochain complex (called the domain complex) in the
category of Hilbert spaces, as depicted in the following diagram:
\begin{equation*}
  \cdots V ^{k-1} \xrightarrow{ \mathrm{d} ^{k-1} } V ^k \xrightarrow{
    \mathrm{d} ^k } V ^{ k + 1 } \cdots .
\end{equation*} 
It follows immediately that $ \mathrm{d} = \bigoplus _k \mathrm{d} ^k
$ is a nilpotent operator on $ W = \bigoplus _k W ^k $, with dense
domain $ V = \bigoplus _k V ^k $. In other words, \emph{a Hilbert
  complex is simply a graded Hilbert space with a graded nilpotent
  operator}. Conversely, if $W$ is a Hilbert space with a nilpotent
operator $ \mathrm{d} $, then it corresponds to the Hilbert complex
with $ W ^k = W $ and $ \mathrm{d} ^k = \mathrm{d} $ for all $ k \in
\mathbb{Z} $, i.e., its domain complex has the infinite diagram
\begin{equation*}
\cdots V  \xrightarrow{ \mathrm{d} } V \xrightarrow{ \mathrm{d} } V
 \cdots .
\end{equation*} 
Additionally, \citet{ArFaWi2010} define some additional structures
that a Hilbert complex may have, corresponding to the different types
of nilpotent operators defined in \autoref{def:nilpotent}. A
``dictionary'' between nilpotent operator terminology and that of
Hilbert complexes is given in \autoref{tab:dictionary}.

\begin{table}
  \begin{center}
    \setlength{\tabcolsep}{1em}
    \begin{tabular}{rl}
      \textbf{ungraded} & \textbf{graded}\\
      nilpotent operator & Hilbert complex\\
      closed-nilpotent operator & closed Hilbert complex\\
      Fredholm-nilpotent operator & Fredholm complex\\
      diffuse Fredholm-nilpotent operator & \parbox[t]{3in}{Hilbert
        complex with \\
        the ``compactness property''}
    \end{tabular}
    \vskip 2ex 
  \end{center}
  \caption{A ``dictionary'' between the language of nilpotent
    operators and that of Hilbert complexes.}
  \label{tab:dictionary}
\end{table}

To discretize the Hodge--Laplace problem on a closed Hilbert complex,
\citet{ArFaWi2010} consider a \emph{Hilbert subcomplex} consisting of
closed subspaces $ V _h ^k \subset V ^k $ such that $ \mathrm{d} _h ^k
= \mathrm{d} ^k \rvert _{ V _h ^k } $, and equipped with bounded
commuting projections $ \pi _h ^k \colon V ^k \rightarrow V _h ^k
$. Again, taking the direct sums $ V _h = \bigoplus _k V _h ^k $, $
\mathrm{d} _h = \bigoplus _k \mathrm{d} _h ^k $, and $ \pi _h =
\bigoplus _k \pi _h ^k $, one gets an approximating subspace and
bounded commuting projection satisfying the conditions of
\autoref{sec:discrete}. In other words, a Hilbert subcomplex with
bounded commuting projections is precisely the situation of
\autoref{sec:discrete}, in the case where the spaces and maps are
graded.

This observation implies that the elements commonly used in finite
element exterior calculus also yield a stable discretization of the
Hodge--Dirac problem. Indeed, these elements---which include the $
\mathcal{P} _r $ and $ \mathcal{P} _r ^- $ families of
piecewise-polynomial differential forms on simplicial meshes
(cf.~\citet{ArFaWi2006,ArFaWi2010}) and the more recent $ \mathcal{S}
_r $ family on cubical meshes (cf.~\citet{ArAw2013})---give
subcomplexes of the $ L ^2 $ de~Rham complex with bounded commuting
projections. Therefore, taking the direct sum over all degrees $k$,
one obtains subspaces and projections satisfying the conditions in
\autoref{sec:discrete}, implying stability and convergence for the
Hodge--Dirac problem.

\section{Numerical application: computing vector fields with
  prescribed divergence and curl}
\label{sec:numerical}

We conclude with a simple numerical example, illustrating an
application of the techniques developed throughout the paper.

Given a bounded domain $ U \subset \mathbb{R}^2 $, let $ L ^2 \Omega
(U) $ be the $ L ^2 $ de~Rham complex on $U$. For the domain of $
\mathrm{d} $, take either the domain complex $ H \Omega (U) $, with
natural boundary conditions, or $ \mathring{H} \Omega (U) $, with
essential boundary conditions, i.e., $0$-forms vanish on the boundary
and $1$-forms vanish tangent to the boundary (cf.~\citet[Sections 4.2
and 6.2]{ArFaWi2010}).  Now, given scalar functions $ f, g \in L ^2
(U) $, consider the Hodge--Dirac problem
\begin{equation*}
  \mathrm{D} u + p = f + g \,\mathrm{d}x _1 \wedge \mathrm{d} x _2 .
\end{equation*} 
Notice that, since there is no $1$-form component on the right-hand
side, the $0$- and $2$-form components of $u$ vanish, so $u$ is simply
a $1$-form. Furthermore, if $f$ and $g$ are orthogonal to harmonic
$0$- and $2$-forms, respectively, then $ p = 0 $ as well. If, in
addition, $U$ is simply connected (so that there are no harmonic
$1$-forms), then $u$ is the unique $1$-form satisfying
\begin{equation*}
  \mathrm{d} ^\ast u = f , \qquad \mathrm{d} u = g \,\mathrm{d}x _1
  \wedge \mathrm{d} x _2 .
\end{equation*} 
Identifying the $1$-form $u = u _1 \,\mathrm{d}x _1 + u _2
\,\mathrm{d}x _2 $ with the vector field $ ( u _1 , u _2 ) $, this is
equivalent to
\begin{equation*}
  - \operatorname{div} u = f , \qquad \operatorname{curl} u = g .
\end{equation*} 
That is, the Hodge--Dirac problem allows us to find a vector field
with prescribed divergence and curl, subject to either natural
boundary conditions ($u$ is tangent to the boundary) or essential
boundary conditions ($u$ is normal to the boundary).

\begin{figure}
  \centering
  \includegraphics[width=0.45\linewidth, clip=true, trim=1.75in .25in
  1.75in .25in]{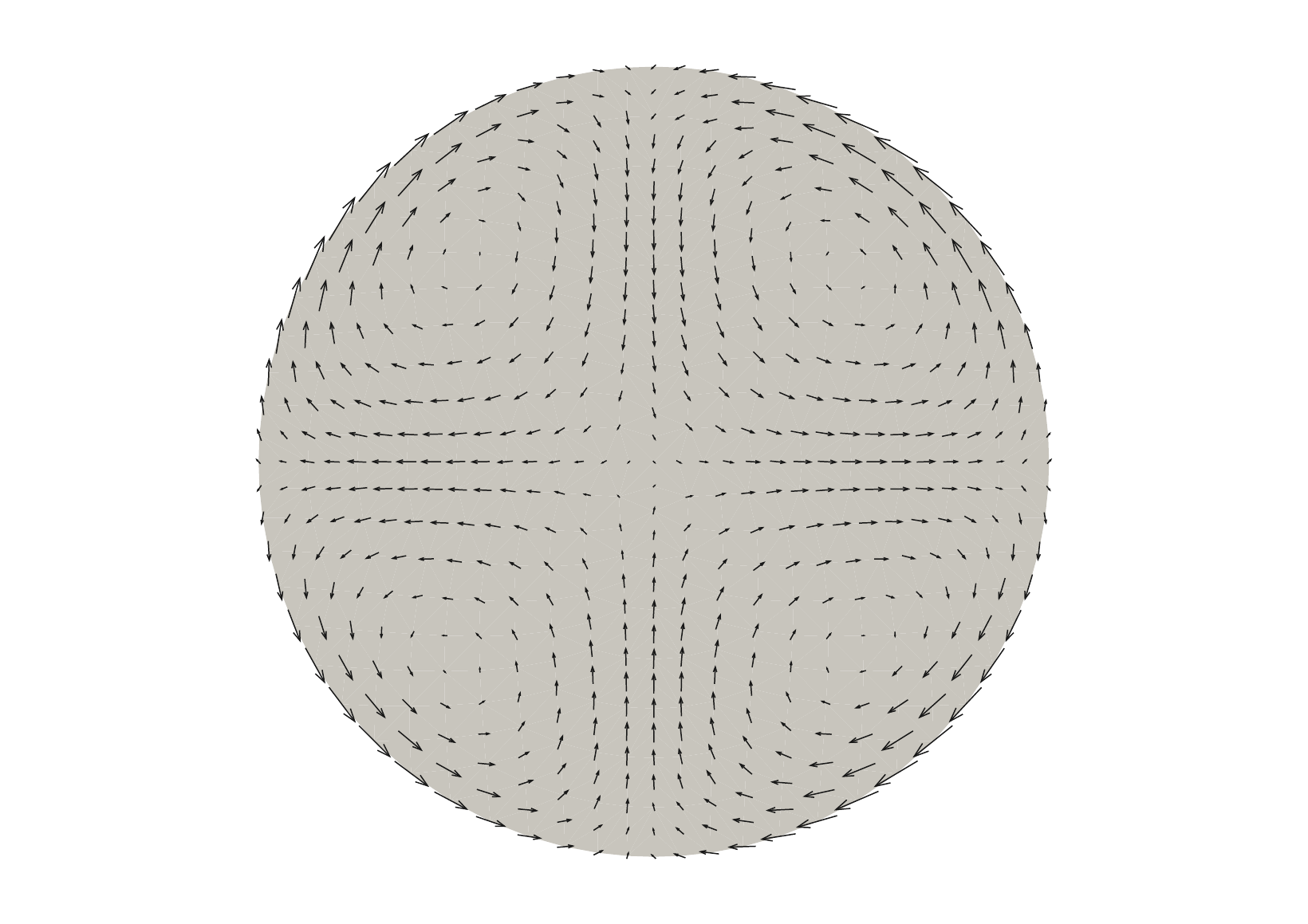} \hfil
  \includegraphics[width=0.45\linewidth, clip=true, trim=1.75in .25in
  1.75in .25in]{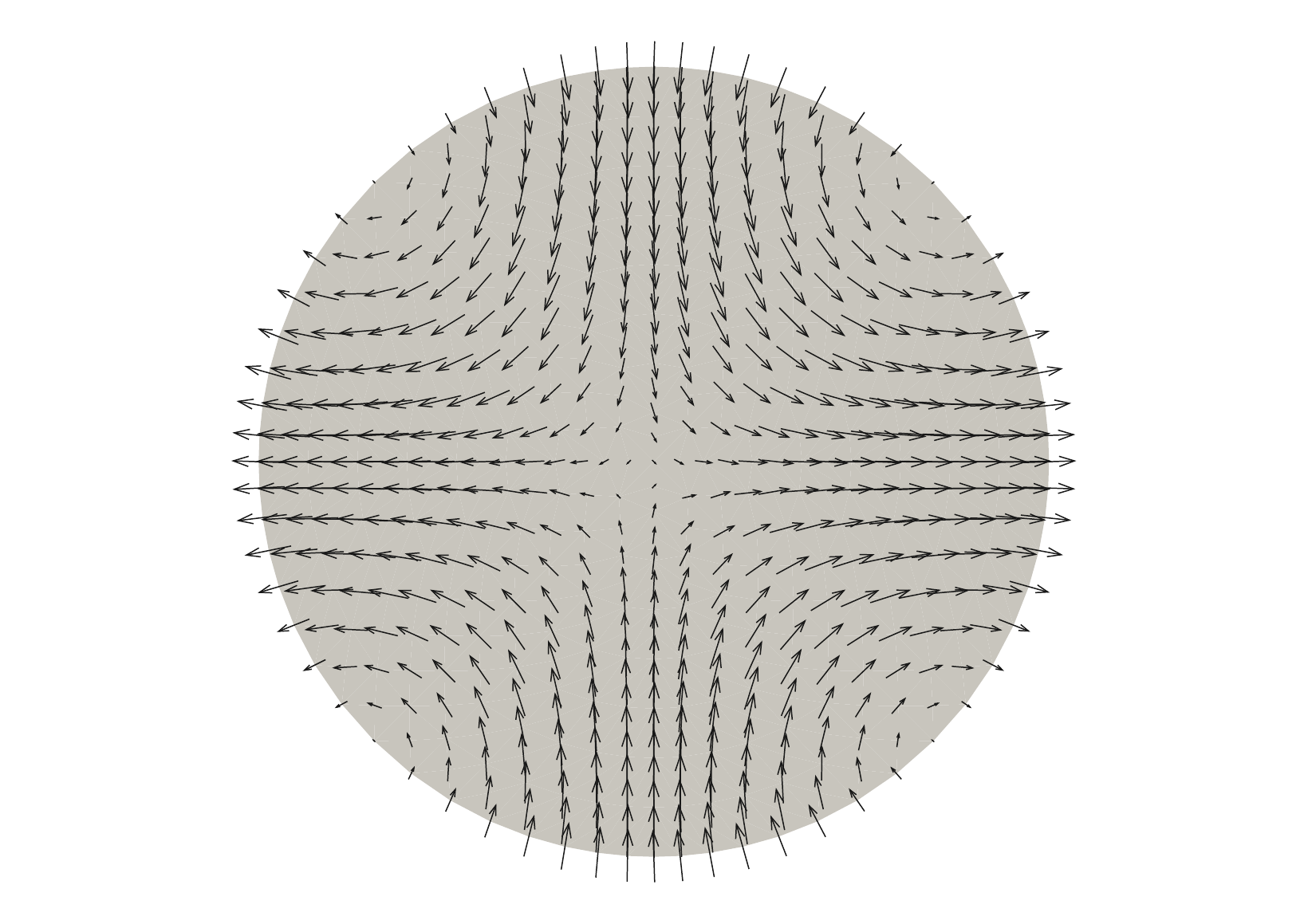}
  \caption{Divergence-free vector fields on the unit disk with curl
    equal to $ x _1 x _2 $. Left: natural boundary conditions
    (vanishing normal component). Right: essential boundary conditions
    (vanishing tangential component).}
  \label{fig:numerical}
\end{figure}

\autoref{fig:numerical} shows numerical solutions to this problem for
both choices of boundary conditions, where $U$ is taken to be the unit
disk, $ f = 0 $, and $ g = x _1 x _2 $. (Since $ \mathfrak{H} $
contains only constant functions, and $ \int _U f = \int _U g = 0 $,
it follows that each is orthogonal to $\mathfrak{H}$.) The discrete
Hodge--Dirac problem \eqref{eqn:discreteHodgeDirac} was implemented
using FEniCS (\citet{LoMaWe2012}), where $ V _h $ was taken to be the
space of $ \mathcal{P} _1 ^- $ differential forms, i.e., Whitney
forms.

Finally, we note that this problem of ``designing'' vector fields has
an interesting application in computer graphics, as studied by
\citet{FiScDeHo2007}. For hair- and fur-like textures, rather than
having an artist painstakingly specify the direction of every
individual hair, it is easier for the artist to specify the divergence
and curl at a set of control points (e.g., cowlicks, whorls), and to
solve for the vector field of hair orientations satisfying the desired
properties.

\section{Conclusion}

We have shown that the general techniques of finite element exterior
calculus, in addition to their previously-studied application to the
abstract Hodge--Laplace problem, may also be applied to an
appropriately defined abstract Hodge--Dirac problem. In addition to
resolving questions in the Clifford analysis community about the
discretization of certain Dirac-type operators, this also yields new
insight into the stability and convergence estimates of
\citet{ArFaWi2010}, which may be recovered as corollaries of the
estimates obtained here. Finally, due to the direct relationship
between Hilbert complexes and nilpotent operators, it follows that
stable elements for the Hodge--Laplace problem are also stable for the
Hodge--Dirac problem, and hence no new or exotic elements need to be
introduced.

\begin{acknowledgments}
  P.~L.~gratefully acknowledges the support of the Australian Research
  Council, the Australian Mathematical Sciences Institute, and the
  Australian National University.  A.~S.~ gratefully acknowledges the
  support of the Simons Foundation through its AMS--Simons Travel
  Grant and Collaboration Grants in Mathematics programs.
\end{acknowledgments}


\footnotesize

\begin{thebibliography}{27}
\providecommand{\natexlab}[1]{#1}

\bibitem[{Arnold and Awanou(2013)}]{ArAw2013}
\textsc{D.~N. Arnold and G.~Awanou}, \emph{Finite element differential forms on
  cubical meshes}, 2013. To appear in Math. Comp.

\bibitem[{Arnold et~al.(2006)Arnold, Falk, and Winther}]{ArFaWi2006}
\textsc{D.~N. Arnold, R.~S. Falk, and R.~Winther}, \emph{Finite element
  exterior calculus, homological techniques, and applications}, Acta Numer., 15
  (2006), pp. 1--155.

\bibitem[{Arnold et~al.(2010)Arnold, Falk, and Winther}]{ArFaWi2010}
\leavevmode\vrule height 2pt depth -1.6pt width 23pt, \emph{Finite element
  exterior calculus: from {H}odge theory to numerical stability}, Bull. Amer.
  Math. Soc. (N.S.), 47 (2010), pp. 281--354.

\bibitem[{Atiyah and Singer(1963)}]{AtSi1963}
\textsc{M.~F. Atiyah and I.~M. Singer}, \emph{The index of elliptic operators
  on compact manifolds}, Bull. Amer. Math. Soc., 69 (1963), pp. 422--433.

\bibitem[{Axelsson et~al.(2006)Axelsson, Keith, and McIntosh}]{AxKeMc2006}
\textsc{A.~Axelsson, S.~Keith, and A.~McIntosh}, \emph{Quadratic estimates and
  functional calculi of perturbed {D}irac operators}, Invent. Math., 163
  (2006), pp. 455--497.

\bibitem[{Axelsson and McIntosh(2004)}]{AxMc2004}
\textsc{A.~Axelsson and A.~McIntosh}, \emph{Hodge decompositions on weakly
  {L}ipschitz domains}, in Advances in analysis and geometry, Trends Math.,
  Birkh\"auser, Basel, 2004, pp. 3--29.

\bibitem[{Babu{\v{s}}ka(1971)}]{Babuska1971}
\textsc{I.~Babu{\v{s}}ka}, \emph{Error-bounds for finite element method},
  Numer. Math., 16 (1971), pp. 322--333.

\bibitem[{Bochev and Hyman(2006)}]{BoHy2006}
\textsc{P.~B. Bochev and J.~M. Hyman}, \emph{Principles of mimetic
  discretizations of differential operators}, in Compatible spatial
  discretizations, vol. 142 of IMA Vol. Math. Appl., Springer, New York, 2006,
  pp. 89--119.

\bibitem[{Bossavit(1988)}]{Bossavit1988}
\textsc{A.~Bossavit}, \emph{Whitney forms: a class of finite elements for
  three-dimensional computations in electromagnetism}, Science, Measurement and
  Technology, IEE Proceedings A, 135 (1988), pp. 493--500.

\bibitem[{Bossavit(1998)}]{Bossavit1998}
\leavevmode\vrule height 2pt depth -1.6pt width 23pt, \emph{Computational
  electromagnetism}, Electromagnetism, Academic Press Inc., San Diego, CA,
  1998. Variational formulations, complementarity, edge elements.

\bibitem[{Brackx et~al.(2009)Brackx, De~Schepper, Sommen, and Van~de
  Voorde}]{BrDeSoVa2009}
\textsc{F.~Brackx, H.~De~Schepper, F.~Sommen, and L.~Van~de Voorde},
  \emph{Discrete {C}lifford analysis: a germ of function theory}, in
  Hypercomplex analysis, Trends Math., Birkh\"auser Verlag, Basel, 2009, pp.
  37--53.

\bibitem[{Brezzi and Fortin(1991)}]{BrFo1991}
\textsc{F.~Brezzi and M.~Fortin}, \emph{Mixed and hybrid finite element
  methods}, vol.~15 of Springer Series in Computational Mathematics,
  Springer-Verlag, New York, 1991.

\bibitem[{Br{\"u}ning and Lesch(1992)}]{BrLe1992}
\textsc{J.~Br{\"u}ning and M.~Lesch}, \emph{Hilbert complexes}, J. Funct.
  Anal., 108 (1992), pp. 88--132.

\bibitem[{Christiansen and Winther(2008)}]{ChWi2008}
\textsc{S.~H. Christiansen and R.~Winther}, \emph{Smoothed projections in
  finite element exterior calculus}, Math. Comp., 77 (2008), pp. 813--829.

\bibitem[{Ciarlet(1978)}]{Ciarlet1978}
\textsc{P.~G. Ciarlet}, \emph{The finite element method for elliptic problems},
  North-Holland Publishing Co., Amsterdam, 1978. Studies in Mathematics and its
  Applications, Vol. 4.

\bibitem[{Cnops(2002)}]{Cnops2002}
\textsc{J.~Cnops}, \emph{An introduction to {D}irac operators on manifolds},
  vol.~24 of Progress in Mathematical Physics, Birkh\"auser Boston Inc.,
  Boston, MA, 2002.

\bibitem[{Delanghe(2001)}]{Delanghe2001}
\textsc{R.~Delanghe}, \emph{Clifford analysis: history and perspective},
  Comput. Methods Funct. Theory, 1 (2001), pp. 107--153.

\bibitem[{Desbrun et~al.(2005)Desbrun, Hirani, Leok, and
  Marsden}]{DeHiLeMa2005}
\textsc{M.~Desbrun, A.~N. Hirani, M.~Leok, and J.~E. Marsden}, \emph{Discrete
  exterior calculus}, 2005. Preprint.

\bibitem[{Dirac(1928)}]{Dirac1928}
\textsc{P.~A.~M. Dirac}, \emph{The quantum theory of the electron}, Proc. R.
  Soc. Lond. A, 117 (1928), pp. 610--624.

\bibitem[{Eastwood and Ryan(2007)}]{EaRy2007}
\textsc{M.~G. Eastwood and J.~Ryan}, \emph{Aspects of {D}irac operators in
  analysis}, Milan J. Math., 75 (2007), pp. 91--116.

\bibitem[{Falk and Winther(2013)}]{FaWi2013}
\textsc{R.~S. Falk and R.~Winther}, \emph{Local bounded cochain projections},
  2013. To appear in Math. Comp.

\bibitem[{Faustino(2009)}]{Faustino2009}
\textsc{N.~Faustino}, \emph{Discrete {C}lifford analysis}, Ph.D. thesis,
  Universidade de Aveiro, 2009.

\bibitem[{Faustino et~al.(2007)Faustino, K{\"a}hler, and Sommen}]{FaKaSo2007}
\textsc{N.~Faustino, U.~K{\"a}hler, and F.~Sommen}, \emph{Discrete {D}irac
  operators in {C}lifford analysis}, Adv. Appl. Clifford Algebr., 17 (2007),
  pp. 451--467.

\bibitem[{Fisher et~al.(2007)Fisher, Schr\"{o}der, Desbrun, and
  Hoppe}]{FiScDeHo2007}
\textsc{M.~Fisher, P.~Schr\"{o}der, M.~Desbrun, and H.~Hoppe}, \emph{Design of
  tangent vector fields}, ACM Trans. Graph., 26 (2007).

\bibitem[{Friedrich(2000)}]{Friedrich2000}
\textsc{T.~Friedrich}, \emph{Dirac operators in {R}iemannian geometry}, vol.~25
  of Graduate Studies in Mathematics, American Mathematical Society,
  Providence, RI, 2000. Translated from the 1997 German original by Andreas
  Nestke.

\bibitem[{Harrison(2005)}]{Harrison2005}
\textsc{J.~Harrison}, \emph{Ravello lecture notes on geometric calculus---{Part
  I}}, 2005. Preprint.

\bibitem[{Hiptmair(1999)}]{Hiptmair1999}
\textsc{R.~Hiptmair}, \emph{Canonical construction of finite elements}, Math.
  Comp., 68 (1999), pp. 1325--1346.

\bibitem[{Hiptmair(2002)}]{Hiptmair2002}
\leavevmode\vrule height 2pt depth -1.6pt width 23pt, \emph{Finite elements in
  computational electromagnetism}, Acta Numer., 11 (2002), pp. 237--339.

\bibitem[{Kotiuga(1984)}]{Kotiuga1984}
\textsc{P.~R. Kotiuga}, \emph{Hodge decompositions and computational
  electromagnetics}, Ph.D. thesis, McGill University, 1984.

\bibitem[{Logg et~al.(2012)Logg, Mardal, Wells et~al.}]{LoMaWe2012}
\textsc{A.~Logg, K.-A. Mardal, G.~N. Wells, et~al.}, \emph{Automated Solution
  of Differential Equations by the Finite Element Method}, Springer, 2012.

\bibitem[{Moisil and Theodoresco(1931)}]{MoTh1931}
\textsc{G.~Moisil and N.~Theodoresco}, \emph{{Fonctions holomorphes dans
  l'espace.}}, Mathematica (Cluj), 5 (1931), pp. 142--159.

\bibitem[{N{\'e}d{\'e}lec(1980)}]{Nedelec1980}
\textsc{J.-C. N{\'e}d{\'e}lec}, \emph{Mixed finite elements in
  {$\mathbb{R}^{3}$}}, Numer. Math., 35 (1980), pp. 315--341.

\bibitem[{N{\'e}d{\'e}lec(1986)}]{Nedelec1986}
\leavevmode\vrule height 2pt depth -1.6pt width 23pt, \emph{A new family of
  mixed finite elements in {$\mathbb{R}^{3}$}}, Numer. Math., 50 (1986), pp.
  57--81.

\bibitem[{Witten(1981)}]{Witten1981}
\textsc{E.~Witten}, \emph{A new proof of the positive energy theorem}, Comm.
  Math. Phys., 80 (1981), pp. 381--402.

\end{thebibliography}
\end{document}